\documentclass [11pt,a4paper]{article}
\usepackage{amssymb, amsmath,latexsym,amsfonts,amsbsy, amsthm}
\usepackage[ansinew]{inputenc}

\usepackage{graphicx}
\usepackage{cite}
\usepackage{color}
\usepackage{mathrsfs}
\usepackage{hyperref}

\newtheorem{theorem}{\textbf{Theorem}}[section]
\newtheorem{lemma}{\textbf{Lemma}}[section]
\newtheorem{proposition}{\textbf{Proposition}}[section]

\newtheorem{remark}{\textbf{Remark}}[section]
\newtheorem{definition}{\textbf{Definition}}[section]
\numberwithin{equation}{section}

\allowdisplaybreaks[4]

\def\non{\nonumber }

\newcommand{\NN}{\mathbb N}

\newcommand\p{{\partial}}
\newcommand\eps{{\varepsilon}}
\newcommand{\RR}{\mathbb{R}}
\newcommand{\QQ}{\mathcal{Q}_{phy}}

\newcommand{\E}{\mathcal E}
\newcommand{\G}{\mathcal G}

\newcommand{\Id}{\mathbb I}

\newcommand{\lm}{\displaystyle\liminf}
\newcommand{\TT}{\mathbb{T}}

\newcommand{\ud}{\mathrm{d}}

\def\({\left (}
\def\){\right )}
\newcommand{\tr}{\operatorname{tr}}
\newcommand{\IO}{\mathrm{I}_0}

\begin{document}

\date{}
\pagestyle{plain}
\title{Regularity of a gradient flow generated by the anisotropic Landau-de Gennes energy with a singular potential}


\author{
  Yuning Liu\footnote{NYU Shanghai, 1555 Century Avenue, Shanghai 200122, China,
and NYU-ECNU Institute of Mathematical Sciences at NYU Shanghai, 3663 Zhongshan Road North, Shanghai, 200062, China
  \texttt{yl67@nyu.edu}}\,
  \and
  Xin Yang Lu\footnote{Department of Mathematical Sciences, Lakehead University,
  	955 Oliver Rd.,
  	Thunder Bay, ON P7B 5E1 Canada.   \texttt{xlu8@lakeheadu.ca}}
  	\footnote{
  	Department of Mathematics and Statistics, McGill University, 805 Sherbrooke St. W.,
  	Montreal, QC H3A0B9, Canada.   \texttt{xinyang.lu@mcgill.ca}
} \,
  \and
  Xiang Xu\footnote{
  Department of Mathematics and Statistics,
  Old Dominion University, Norfolk, VA 23529, USA.
  \texttt{x2xu@odu.edu}}
 }
\maketitle
\begin{abstract}
In this paper we study a gradient flow generated by the Landau-de Gennes free energy that describes nematic liquid crystal
configurations in the space of $Q$-tensors. This free energy density functional is composed of three quadratic terms as the elastic
energy density part, and a singular potential in the bulk part that is considered as a natural enforcement of a physical constraint on
the eigenvalues of $Q$. The system is a non-diagonal parabolic system with a singular potential which trends to infinity logarithmically
when the eigenvalues of $Q$ approach the physical boundary. We give a rigorous proof that for rather general initial data with possibly
infinite free energy, the system has a unique strong solution after any positive time $t_0$. Furthermore, this unique strong solution
detaches from the physical boundary after a sufficiently large time $T_0$. We also give estimate of the Hausdorff measure of the set
where the solution touches the  physical boundary and thus prove a partial regularity result of the solution in the intermediate stage
$(0,T_0)$.

\end{abstract}



\section{Introduction}

The Landau-de Gennes theory is a continuum theory of nematic liquid crystals \cite{dG93}.
When formulating static or dynamic continuum theories a crucial step is to select an appropriate order parameter that captures the microscopic
structure of the  rod-like molecule systems.  In our framework the order parameter is a matrix-valued function that takes values in the following  so called $Q$-tensor space
\begin{equation}\label{Q-tensor-space}
\mathcal{Q} := \Big\{M\in\RR^{3\times 3}\big| \tr M=0; \,M=M^T\Big\}.
\end{equation}
It is considered as a suitably normalized second order moment of the
probability distribution function that dictates locally preferred orientations of nematic molecular  directors (cf.
\cite{B12, MZ10, NMreview}).

To formulate the problem, let $\TT^n$ be  the unit box/square in $\mathbb{R}^n$ with $n=2$ or $3$. For each  order parameter $Q: \TT^n\to \mathcal{Q}$,  the associated free energy functional $\E(Q)$ consists of the elastic
and the bulk parts, which reads
\begin{equation}\label{free-energy}
\E(Q)  :=  \G(Q)+\mathcal{B}(Q)-\alpha\|Q\|^2_{L^2(\TT^n)}.
\end{equation}
Here  $\G$ stands for the {anisotropic} elastic
energy that contains three quadratic terms of $\nabla{Q}$:
\begin{equation}\label{elastic-energy}
\G(Q) :=
\begin{cases}\displaystyle
\int_{\TT^3}\big(L_1\p_k Q_{ij} \p_k Q_{ij}+L_2\p_j Q_{ik}\p_k Q_{ij}+L_3\p_j Q_{ij}\p_k Q_{ik}\big)\,\ud
x, \quad\mbox{if } Q\in H^1(\TT^n), \\
+\infty,
\qquad\qquad\qquad\qquad\qquad\qquad\qquad\qquad\qquad\qquad\qquad\qquad\;\mbox{otherwise},
\end{cases}
\end{equation}
where $L_1, L_2, L_3$ are material dependent constants. Here and in the sequel $\partial_kQ_{ij}$ denotes the $k$-th spatial partial derivative of the $ij$-th component
of $Q$, and we adopt Einstein
summation convention by summing over repeated Latin letters.
Following \cite{LW16}, we assume
\begin{equation}\label{coefficient-assumption}
L_1 > 3|L_2+L_3|,
\end{equation}
which ensures that \eqref{elastic-energy} fulfills the strong Legendre condition.

Further,
$\mathcal{B}(Q)$ denotes the bulk energy
$$
  \mathcal{B}(Q) := \int_{\TT^3}\psi(Q)\,\ud{x},
$$
where the integrand $\psi(Q)$  is the singular potential introduced in \cite{BM10}:
\begin{equation}\label{singular-potential-def}
\psi(Q) := \begin{cases}
\displaystyle\inf_{\rho\in\mathcal{A}_{Q}}\int_{\mathbb{S}^2}\rho(p)\ln{\rho(p)}\,\ud p,
\quad\mbox{if } -\frac13<\lambda_i(Q)<\frac23,\; 1\leq i\leq 3,\\
+\infty, \qquad\qquad\qquad\qquad\quad\mbox{otherwise}.
\end{cases}
\end{equation}
Here $\lambda_i(Q)$ denotes the $i$-th eigenvalue of the matrix $Q$ and  $\mathcal{A}_{Q}$  is the admissible class  defined by
$$
  \mathcal{A}_{Q}=\left\{\rho(p):\mathbb{S}^2\rightarrow\overline{\mathbb{R}^+}~\big|~ \,
  \|\rho\|_{L^1(\mathbb{S}^2)}=1;\;
   \int_{\mathbb{S}^2}\big(p\otimes{p}-\frac13\Id\big)\rho(p)\,\ud{p}=Q\right\}.
$$
It is noted that the singular potential \eqref{singular-potential-def} imposes {physical constraints} on the eigenvalues of $Q$.
Futher, $\alpha>0$ in \eqref{free-energy} is a temperature dependent constant which characterizes the relative intensity of the molecular Brownian motion and the molecular interaction \cite{BM10}. We refer interested readers to \cite{B12, BM10, BM20, B18, FRSZ15, LXZ20}
 for detailed discussions of basic analytic properties of $\psi$, such as convexity, smoothness
in its effective domain, blow-up rates near the physical boundary, etc. Meanwhile, various problems in static and dynamic configurations concerning $\psi$ can be found in \cite{BP16, DHW19, EKT16, FRSZ14, FRSZ15, GT19, W12}. Specifically, the free energy in related dynamic problems considered so far in the existing literature \cite{DHW19, FRSZ14, FRSZ15, W12} only involves the
$L_1$ isotropic term. Therefore, we are motivated to study the dynamic problem whose free energy contains anisotropic $L_2, L_3$ terms.
It is worth pointing out that the presence of such terms is more than a mere
technical challenge, since
they make it impossible to recover any kind of maximum principle, which
was crucial in \cite{W12}.

This paper is concerned with a rigorous study of the  gradient flow
generated by $\E(Q)$ in the Hilbert space $L^2(\TT^n;\mathcal{Q})$:
\begin{equation}\label{grad flow E}
\left\{
\begin{array}{rlr}
\partial_tQ(t,\cdot)&\in -\partial\E(Q(t,\cdot)),  &t>0,\\
Q(0, x)&=Q_0(x), &x\in\TT^n
\end{array}
\right.
\end{equation}
subject to periodic boundary condition
\begin{equation}\label{BC-periodic}
Q(t, x+e_i)=Q(t, x), \quad\mbox{for } (t, x)\in\mathbb{R}^+\times\partial\TT^n.
\end{equation}
Here in  \eqref{grad flow E}, $\partial\E(Q)$ is formally the variation of the free energy
\eqref{free-energy}. However, due to the singular feature of $\psi(Q)$, it  should be understood as sub-differential  (see Lemma
\ref{lemma-subgradient} for more details.)


Parallel to the $Q$-tensor space  \eqref{Q-tensor-space},
we  introduce  the {\it physical}
$Q$-tensor space by
\begin{equation}\label{q-tensor}
  \QQ  := \Big\{M\in\mathcal{Q}\big|\,  -\frac13< \lambda_1(M)\leq \lambda_2(M)\leq \lambda_3(M)< \frac23\Big\},
\end{equation}
where $\lambda_i(M)$ denotes the $i$-th eigenvalue of the matrix $M$, ordered  non-decreasingly. Any element in $\QQ$ is called  a physical $Q$-tensor.

Our first main result ensures the existence and uniqueness of solutions to the gradient flow \eqref{grad flow E} with rather general initial data.
\begin{theorem}\label{main-theorem-1}
Let $n=2$ or $3$. For any initial data
\begin{equation}
Q_0\in\overline{\{Q\in L^2(\TT^n;\QQ)\mid \mathcal{E}(Q)<\infty\}}^{L^2(\TT^n)},
\end{equation}
there exists a unique
global solution  $Q(t,x):\RR^+\times\TT^n\to \QQ$ of \eqref{grad flow E} such that
\begin{align}\label{strong solu}
\partial_tQ_{ij}&=2L_1\Delta{Q}_{ij}+2(L_2+L_3) \p_j\p_kQ_{ik}-\frac 23(L_2+L_3) \p_k\p_
{\ell}Q_{\ell k}\delta_{ij}\non\\
&\qquad-\frac{\partial\psi}{\partial{Q}_{ij}}+\frac13\tr\Big(\frac{\partial\psi}{\partial{Q}}\Big)\delta_{ij}+2\alpha Q_{ij}
\end{align}
holds almost everywhere in $(0,\infty)\times \TT^n$. And for any fixed $t_0>0$, the solution satisfies
\begin{equation}
Q\in L^\infty(t_0,\infty; H^1(\TT^n)),\qquad \partial_tQ\in L^2_{loc}(t_0,\infty; L^2(\TT^n)),
\end{equation}
and  the energy dissipative equality
\begin{equation}\label{energy identity}
\int_{t_0}^T\big(\|\partial_tQ(t,\cdot)\|_{L^2(\TT^3)}^2+\|\partial\E(Q(t,\cdot))\|^2_{L^2(\TT^3)}\big)\,\ud{t}
= 2\E(Q(t_0))-2\E(Q(T))
\end{equation}
for all $0<t_0<T<+\infty$. Further, $Q(t,\cdot)$ is physical in the sense that
\begin{equation}\label{strict-physicality}
 Q(t, x)\in\QQ, \quad \forall t>0, \; a.e.\; x\in\TT^n.
\end{equation}
\end{theorem}

It is worthy to point out that due to the energy dissipative property of the gradient flow as well as the convexity of the singular potential, for any $T>0$ one can formally establish the a priori estimate  of $Q$ in $L^\infty_{loc}(0, T; H^1(\TT^n))\cap L^2_{loc}(0, T; H^2(\TT^n))$. As a consequence, existence of weak solutions to \eqref{grad flow E} can be achieved by using two level
approximation schemes as in \cite{W12}, i.e., regularizing the initial data and the singular free energy. However, such
arguments involve fairly complicated approximation procedures. Fortunately Ambrosio--Gigli--Savar{\'e} \cite{AGS08} provides a powerful framework to obtain the
solution under very general assumptions of the initial data.


To establish higher regularity, namely a uniform-in-time $H^2$ bound of
the solution $Q$,  essential difficulties arise from the anisotropic  terms. Without the $L_2+L_3$ terms,
 the convexity of $\psi$ as well as the classical $L^1-L^\infty$ estimate of heat equation ensure the
strict physicality at any positive time (see section $8$ in \cite{W12} for details) and henceforth conventional energy method applies.
Concerning the gradient flow \eqref{grad flow E}, unfortunately the anisotropic terms make such maximum
principle argument invalid. As a consequence, the proof of higher regularity of the solution becomes quite subtle in the sense that
$Q$ might not stay inside any compact subset of $\QQ$. To overcome such a difficulty, we need to make a careful exploitation of its gradient flow structure, as well as to combine several results on the gradient flow theory given in \cite{AGS08} and Gamma-convergence of gradient flows discussed in \cite{S11, SS04}.
These lead to the next theorem which  improves the regularity by establishing  the uniform-in-time $H^2$ bound of the unique solution to the gradient flow \eqref{grad flow E}. Further, it can be shown that this unique solution detaches from its physical boundary after a sufficiently large time $T_0$.

\begin{theorem}\label{main-theorem-2}
For any $t_0>0$ the  solution  established in Theorem
\ref{main-theorem-1} enjoys the improved regularity $Q\in
L^\infty(t_0,+\infty; H^2(\TT^n))$, and for almost every $t\geq t_0$ there holds
\begin{equation}\label{uniform-H2-bound}
\|\Delta Q(t,\cdot)\|_{L^2(\TT^n)}\leq C_L\left(e^{4\alpha}\sqrt{\E(Q(t_0))-\inf\E+1}+2\alpha\|Q(t,\cdot)\|_{L^2(\TT^n)}\right),
\end{equation}
where $C_L$ is expressed by
\begin{equation}\label{C-L}
C_L := \frac{1}{2(L_1-|L_2+L_3|)}\sqrt{\frac{L_1+|L_2+L_3|}{L_1+|L_2+L_3|-2\sqrt{L_1|L_2+L_3|}}}
\end{equation}
Furthermore, under the stronger assumption
\begin{equation}\label{coefficient-assumption-3}
L_1-3|L_2+L_3| -\alpha{C}_{\TT^n}^2 >0,
\end{equation}
where $C_{\TT^n}=(2\pi)^n$ is the Poincar\'e constant in $\TT^n$,
there exists $T_0>0$ such that the unique solution
is strictly physical
for all $t\geq T_0$  in the sense that
\begin{equation}\label{uniform-physicality}
-\dfrac13+\kappa\leq\lambda_i(Q(t, x))\leq\dfrac23-\kappa,\quad\forall x\in\TT^n
\end{equation}
for some  constant $\kappa\in (0,1/6)$.
\end{theorem}

During the period  $(0, T_0)$,  a partial regularity result of the unique solution can be established, i.e. Hausdorff  dimension
of the set where the solution touches the physical boundary $\partial\QQ$:
\begin{theorem}\label{thm-hausdorff}
Let $Q(t, x)$
be the unique strong solution of \eqref{grad flow E} established in Theorem \ref{main-theorem-2}. Then for a.e. $t\in (0, T_0)$, the contact set
\begin{equation}
\Sigma_t:=\{x\in\TT^n\mid Q(t,x)\in \p \QQ \}
\end{equation}
has the following estimate:
\begin{itemize}
\item $\mbox{dim}_{\mathcal{H}}(\Sigma_t)\leq 2$ for $n=3$.
\item $\mbox{dim}_{\mathcal{H}}(\Sigma_t)= 0$ for $n=2$.
\end{itemize}
\end{theorem}

The rest of the paper is organized as follows. Some notations and preliminaries are provided in Section 2. The proofs of the three main results, namely Theorems
\ref{main-theorem-1}, \ref{main-theorem-2}, and \ref{thm-hausdorff},
are given in Sections 3, 4, 5, respectively.

\section{Preliminaries}



We start with a few basic notations in $Q$-tensor theory.
For any $Q\in\mathcal{Q}$,  $|Q|:=\sqrt{\tr(Q^tQ)}$ represents the Frobenius norm of $Q$.
The gradient of the function $\psi(Q)$ will be abbreviated by $\psi'(Q)$, and its components are denoted by $\psi'_{ij}(Q):=\frac{\p \psi(Q)}{\p Q_{ij}}$.
Moreover, we denote $L^2(\TT^n;\mathcal{Q})$ the Hilbert space endowed
with the $L^2$ metric
$$
  \|Q\|_{L^2(\TT^n)}=\sqrt{\int_{\TT^n}\mbox{tr}(Q^tQ)}=\sqrt{\int_{\TT^n}\mbox{tr}(Q^2)},
  \quad\mbox{for } Q: \TT^n\rightarrow\mathcal{Q}.
$$
Here and after, for brevity, $\|\cdot\|_{L^2(\TT^n)}$ will often be written as
$\|\cdot\|_{L^2}$, or simply $\|\cdot\|$.

\smallskip

Next we provide some preliminaries of gradient flow theory in a Hilbert space.
We start with some basic definitions  in a Hilbert space $H$, with inner product $\langle\cdot, \cdot\rangle$ and norm $\|\cdot\|$ (cf. \cite{B11, EV98}).
\begin{definition}
A function $f: H\rightarrow \RR\cup \{+\infty\}$ is called proper if $f$ is not identically equal to $+\infty$. The effective
domain of $f$ is defined by
$$
  D(f)=\big\{u\in H| \,f(u)<+\infty\big\}.
$$
\end{definition}
 By \cite{BM10}, the effective domain $D(\psi )$ is equivalent to
\eqref{singular-potential-def}.
\begin{definition}
Let $\lambda\in\RR$, a $\lambda$-convex function $F: H\rightarrow (-\infty, +\infty]$ is a function satisfying
$$
  F((1-t)u+tv)\leq (1-t)F(u)+tF(v)-\frac{\lambda}{2}t(1-t)\|u-v\|^2, \qquad\forall u, v\in H.
$$
For each $u\in H$,   $\partial{F}[u]$ is defined as the set of $w\in H$ such that
$$
  F(u)+\langle w, v-u \rangle+\frac{\lambda}{2}\|u-v\|^2\leq F(v), \quad\forall v\in H.
$$
The mapping $\partial{F}: H\rightarrow 2^{H}$ is called the subdifferential of $F$. Further,
We say $u\in D(\partial F)$, the domain of $\partial F$, provided $\partial{F}[u]$ is not empty.
\end{definition}

\begin{definition}
We say $u(t)$ is a gradient flow of $F$ starting from $u_0\in H$ if it is a locally absolutely continuous
curve in $(0, +\infty)$ such that
\begin{equation}\label{gradient-flow-def}
\begin{cases}
\partial_tu(t)\in-\partial F(u(t)), \quad \mbox{a.e.}\; t>0\\
\displaystyle\lim_{t\rightarrow 0^+}u(t)=u_0.
\end{cases}
\end{equation}
\end{definition}

The next result   is   due to  \cite[Theorem 4.0.4]{AGS08} which was originally  stated under  metric space setting. For the purpose of proving Theorem \ref{main-theorem-1}, it suffices to  rewrite it in the Hilbert space setting:
\begin{proposition}\label{theorem-AGS}
Let $\lambda\in\RR$ and $F: H\rightarrow (-\infty, +\infty]$ be a proper,
bounded from below, and lower semicontinuous functional. 
Suppose moreover that $F$ is $\lambda$-convex, which, since we work
in an Hilbert space, is equivalent to
assuming that, for each $ \tau\in (0,\frac{1}{\lambda^-})$ with  $\lambda^{-} := \max\{0,-\lambda\}$ and each fixed $w\in D(F)$,  
the functional
\begin{equation}
 \Phi(\tau, w; v)=\frac{1}{2\tau}\|v-w\|^2+F(v),~\forall v\in D(F)
\end{equation}
satisfies the following inequality for every $v_0,v_1\in D(F)$:
\begin{align*}
\Phi(\tau, w; t)\leq (1-t)\Phi(\tau,w;v_0)+t\Phi(\tau,w;v_1)-\frac{1+\lambda\tau}{2\tau}t(1-t)\|v_0-v_1\|^2.
\end{align*}
 Then for each $u_0\in \overline{D(F)}$, $u(t)=\displaystyle\lim_{k\rightarrow+\infty}J_{t/k}^k(u_0)$ with
$J_\tau$ being the resolvent
\begin{equation}
  X\in J_\tau(Y) \Longleftrightarrow X\in\mbox{argmin}\Big\{F(\cdot)+\frac{1}{2\tau}\|Y-\cdot\|^2
  \Big\},
\end{equation}
satisfies
\begin{enumerate}
\item Variational inequality: $u$ is the unique solution to the
evolution variational inequality
\begin{equation}\label{variational-inequality}
\frac12\dfrac{d}{dt}\|u(t)-v\|^2+\frac{\lambda}{2}\|u(t)-v\|^2+F(u(t))\leq F(v),
\qquad \mbox{for a.e. } t>0 \mbox{ and } v\in D(F),
\end{equation}
among all the locally absolutely continuous curves such that
$u(t)\rightarrow u_0$ as $t \downarrow 0^+$.

\item Regularizing effect: $u$ is locally Lipschitz regular, and
$u(t,\cdot)\in D(\E)$ for all $t>0$.
\end{enumerate}
\end{proposition}
\begin{remark}\label{remark-equivalence}
It is well known that for a $\lambda$-convex function $F: H\rightarrow (-\infty, +\infty]$, a locally absolutely continuous curve
$u(t)$ in $(0, +\infty)$ satisfies \eqref{gradient-flow-def} if and only if it satisfies the evolution variational inequality \eqref{variational-inequality}.
\end{remark}

Now we turn to the $\Gamma$-convergence of gradient flows in a Hilbert space, a theory developed in   \cite{SS04} and  \cite{S11}.
Let $\{u_n\}$ be the solution to the gradient flow
\begin{equation}\label{GF u_n}
\partial_tu_n=-\nabla E_n(u_n)
\end{equation}
of a $C^1$ functional sequence $\{E_n\}$. Assume $E_n$ $\Gamma$-converges to a functional $F$, and  there is a general sense of convergence $u_n\overset{S}{\rightarrow}  u$, relative to which the $\Gamma$-convergence of $E_n$ to $F$ holds. We introduce the ``energy-excess"
along a family of curves $u_n(t)$ with $u_n(t)\overset{S}{\rightarrow}u(t)$ by setting
$$
  \tilde{D}(t)=\limsup_{n\rightarrow\infty}E_n(u_n(t))-F(u(t)).
$$
The main result of \cite {S11} is the following:
\begin{proposition}\label{theorem-serfaty}
Assume $E_n$ and $F$ satisfy  a $\Gamma-\liminf$ relation: if $u_n\overset{S}{\rightarrow} u$ as $n\rightarrow\infty$ then
$$
  \displaystyle\liminf_{n\rightarrow\infty}E_n(u_n)\geq F(u).
$$
Assume that the following two additional conditions hold:
\begin{enumerate}
\item (Lower bound on the velocities) If $u_n(t)\overset{S}{\rightarrow} u(t)$ for all $t\in [0, T)$ then there exists
$f\in L^1(0, T)$ such that for every $s\in [0, T)$
\begin{equation}\label{Serfaty-condition1-1}
\liminf_{n\rightarrow\infty}\int_0^s\|\partial_tu_n(t)\|_{H}^2\,dt \geq \int_0^s\big[\|\partial_tu(t)\|_H^2-f(t)\tilde{D}(t)\big]\,dt.
\end{equation}
\item (Lower bound for the slopes) If $u_n\overset{S}{\rightarrow} u$ then
\begin{equation}\label{Serfaty-condition1-2}
\liminf_{n\rightarrow\infty}\|\nabla E_n(u_n)\|_{H}^2\,dt\geq\|\nabla F(u)\|_H^2-C\tilde{D},
\end{equation}
where $C$ is a universal constant, and $\|\nabla F(u)\|$ denotes the  minimal norm of the elements in $\partial F(u)$.
\end{enumerate}
Assume $u_n(t)$ is  a family of solutions to \eqref{GF u_n} on $[0, T)$ with $u_n(t)\overset{S}{\rightarrow} u(t)$ for all $t\in [0, T)$,
such that
$$
  E_n(u_n(0))-E_n(u_n(t))=\int_0^t\|\partial_tu_n(s)\|_{H}^2\,ds, \quad\forall t\in [0, T).
$$
Assume also that
$$
 \displaystyle\lim_{n\rightarrow\infty}E_n(u_n(0))=F(u(0)),
$$
then $u\in H^1(0, T; H)$ and is a solution to $\partial_tu\in -\partial F(u)$ on $[0, T)$. Moreover, $\tilde{D}(t)=0$ for all $t$ (that is the solutions ``remain well-prepared")
and
$$
  \|\partial_tu_n\|_{H} \xrightarrow{n\to \infty} \|\partial_tu\|_{H}, \quad \|\nabla E_n(u_n)\|_{H}\xrightarrow{n\to \infty} \|\nabla F(u)\|_{H}\quad\mbox{in }\; L^2(0, T).
$$
\end{proposition}

\section{Proof of Theorem \ref{main-theorem-1}: Existence of  solutions}

This section is devoted to the proof of Theorem \ref{main-theorem-1}.
First of all, with the choices  $H=L^2(\TT^n;\mathcal{Q})$ and $F=\E$ (defined in
\eqref{free-energy}),
we show that the assumptions in Proposition \ref{theorem-AGS} are satisfied, so that there exists  a unique solution $Q(t,\cdot)$
in variational inequality setting \eqref{variational-inequality} (see Proposition \ref{theorem-MMS} below). Moreover, since the free energy is $-2\alpha$ convex, the solution achieved in Proposition \ref{theorem-MMS} is equivalent to the solution of the gradient flow \eqref{grad flow E} in sub-differential setting. As a consequence,
we compute explicitly the sub-differential of $\E$, and obtain a unique strong
solution to equation \eqref{strong solu}. Finally, we apply two theorems in \cite{AGS08} to show further regularity properties of $Q$ in Theorem \ref{main-theorem-1}.
Since all the following arguments are valid for both $\TT^3$ and $\TT^2$ with minor modifications, for brevity we discuss the case of $\TT^3$ only.

In this subsection we consider the following settings:
$$
  (H,\|\cdot\|)=L^2(\TT^3;\mathcal{Q}), \qquad F=\E(Q).
$$
To begin with, we need to verify all assumptions in Proposition \ref{theorem-AGS} are valid, which is given in the following two lemmas.

\begin{lemma}\label{lemma-lsc}
The free energy functional $\E$ is proper, bounded from below, $-2\alpha$ convex and lower
semicontinuous  in $L^2(\TT^3;\mathcal{Q})$.

\end{lemma}
\begin{proof}

First we show that the elastic energy $\G$ is nonnegative, convex, and lower semicontinuous in
$L^2(\TT^3;\mathcal{Q})$.
It is proved in \cite{LW16} that  when $L_1>0,L_1+L_2+L_3>0$,    $\G$ satisfies the strong  Legendre condition, which implies the convexity of $\G$.
It suffices to show that  $\G$ is nonnegative when $Q\in H^1(\TT^3)$, which follows from the coefficient assumption \eqref{coefficient-assumption}, integration by parts and the Cauchy-Schwarz inequality:
\begin{align}\label{coercivity}
\G(Q)&\geq\int_{\TT^3}\Big(L_1\partial_kQ_{ij}\partial_kQ_{ij}+L_2\partial_kQ_{ik}\partial_kQ_{ij}+L_3\partial_jQ_{ij}\partial_kQ_{ik}\Big)\,dx\nonumber\\
&=\int_{\TT^3}\Big[L_1\partial_kQ_{ij}\partial_kQ_{ij}+(L_2+L_3)\partial_jQ_{ik}\partial_kQ_{ij}\Big]\,dx\nonumber\\
&\geq\int_{\TT^3}\big(L_1-3|L_2+L_3|\big)|\nabla{Q}|^2\,dx\geq 0.
\end{align}
Besides, since $\G$ is convex and quadratic, it is lower semicontinuous \cite[Theorem 8.1]{GM12}.

Next we show that the functional $\mathcal{B}$ is convex, bounded from below, and lower semicontinuous in $L^2(\TT^3;\mathcal{Q})$.
The convexity of $\mathcal{B}$ follows from \cite{BM10, FRSZ15}. A lower bounded can be derived from the inequality  $x\ln{x}\geq -1/e$ for any $x\geq 0$:
\begin{equation}
\mathcal{B}=\int_{\TT^3}\psi (Q)\,\ud{x} \geq
-4\pi^2|\TT^3|/e.
\end{equation}

To show the lower  semicontinuity of $\mathcal{B}$,  let $Q_n\rightarrow Q$ strongly in
$L^2(\TT^3)$. If $\lm_{n\rightarrow\infty}\psi (Q_n)=+\infty$ on a set of positive measure, then
the proof is done. Thus upon subsequence we assume
\begin{equation}\label{sequence-convergence}
  \lm_{n\rightarrow\infty}\mathcal{B}(Q_n)=\displaystyle\lim_{n\rightarrow\infty}\mathcal{B}(Q_n)<+\infty,
  \;\;\mbox{and}\;\; Q_n(x) \xrightarrow{n\to \infty} Q(x) \;\;\mbox{for a.e. }
  x\in\TT^3.
\end{equation}
Consequently, for all $n\in\NN$ sufficiently large
and a.e. $x\in\TT^3$, all eigenvalues of $Q_n(x)$ are in $(-1/3, 2/3)$.
Moreover, the eigenvalues of $Q(x)$ are in $[-1/3, 2/3]$ since
convergence of eigenvalues  follows from   convergence of the matrices (cf. \cite{R69}).

\smallskip

We claim that for a.e. $x\in\TT^3$ the eigenvalues of $Q(x)$ are in
$(-1/3, 2/3)$. To this aim, we argue by contradiction.
Assume the opposite, i.e.
$E=\big\{x\in\TT^3, \lambda_1(Q(x))=-1/3\big\}$
has positive measure. Then it follows from \cite{BM10} that
$\psi (Q_n(x))\xrightarrow{n\to\infty} +\infty$ in $E$, and henceforth Fatou's
lemma implies
\begin{align*}
\displaystyle\liminf_{n\rightarrow\infty}\mathcal{B}(Q_n)&\geq
\displaystyle\liminf_{n\rightarrow\infty}\int_{E}\psi (Q_n)\,\ud{x}
+\displaystyle\liminf_{n\rightarrow\infty}\int_{\TT^3\setminus E}\psi (Q_n)\,\ud{x}\\
&\geq\int_{E}\displaystyle\liminf_{n\rightarrow\infty}\psi (Q_n)\,\ud{x}-\dfrac{4\pi^2}{e}|\TT^3\setminus E|
=+\infty,
\end{align*}
which is in contradiction with \eqref{sequence-convergence}. Thus the claim is proved.
Since $\psi $ is smooth in $D(\psi)=\QQ$ (see
\cite{FRSZ15}), we have $\psi (Q_n(x))\xrightarrow{n\to\infty}\psi (Q(x))$ for
a.e. $x\in\TT^3$. Thus Fatou's lemma implies
\begin{align*}
\displaystyle\liminf_{n\rightarrow\infty}\mathcal{B}(Q_n)
\geq\int_{\TT^3}\displaystyle\liminf_{n\rightarrow\infty}\psi (Q_n)\,\ud{x}
=\int_{\TT^3}\psi (Q(x))\,\ud{x}=\mathcal{B}(Q).
\end{align*}

To sum up, $\E$ is $-2\alpha$ convex, and lower semicontinuous in $L^2(\TT^3;\mathcal{Q})$. It remains to show $\E$ is proper and bounded from below.
Clearly, $\E(Q)<+\infty$ provided $Q\in H^1(\TT^3)$ and $Q(x)\in D(\psi )$ a.e. $x\in\TT^3$, hence $\E$ is proper. Further, if $Q$ is not physical
then $\E(Q)=+\infty$ , while if $Q$ is physical then
$\|Q\|_{L^2(\TT^3)}$ is bounded, hence $\E$ is bounded from below
since both $\G$ and $\psi$ are bounded from below.

\end{proof}

\begin{remark}
It is noted that the coefficient assumption \eqref{coefficient-assumption} is different from the one in \cite{DG98, KRSZ16}
which ensures the elastic energy $\G$ is coercive only.
\end{remark}

\begin{lemma}\label{prop-convexity}
For any $R, P_0, P_1\in D(\E)$, denote $\gamma_t=(1-t)P_0+tP_1$,
$t\in [0, 1]$, then for
each $0<\tau<1/2\alpha$ the functional
\begin{equation}
  \Phi(\tau, R; \gamma_t) :=
  \dfrac{\|\gamma_t-R\|^2}{2\tau}+\E(\gamma_t)
\end{equation}
is $(1/\tau-2\alpha)$-convex on $\gamma_t$, in the sense that
\begin{align}
\Phi&(\tau, R; \gamma_t)\non\\
&\leq (1-t)\Phi(\tau, R; P_0)+t\Phi(\tau, R;
P_1)-\dfrac{(1-2\alpha\tau)}{2\tau}t(1-t)\|P_1-P_0\|^2, \quad\forall t\in [0, 1].
\end{align}

\end{lemma}
\begin{proof}
We infer from the convexity of $\G$ and $\mathcal{B}$ that
\begin{align*}
\Phi(\tau, R;\gamma_t)
&=\dfrac{\big\|(1-t)P_0+tP_1-R\big\|^2}{2\tau}+(\G+\mathcal{B})((1-t)P_0+tP_1)-\alpha\big\|(1-t)P_0+tP_1\big\|^2\\
&\leq\dfrac{(1-t)\|P_0-R\|^2+t\|P_1-R\|^2-t(1-t)\|P_0-P_1\|^2}{2\tau}+(1-t)(\G+\mathcal{B})(P_0)\\
&\qquad+t(\G+\mathcal{B})(P_1)-\alpha\big[(1-t)\|P_0\|^2+t\|P_1\|^2-t(1-t)\|P_0-P_1\|^2\big]\\
&=(1-t)\Phi(\tau, R; P_0)+t\Phi(\tau, R;
P_1)-\dfrac{(1-2\alpha\tau)}{2\tau}t(1-t)\|P_1-P_0\|^2
\end{align*}
\end{proof}

To sum up, we
manage to verify that all assumptions of Proposition \ref{theorem-AGS}
are satisfied, which leads to the following theorem.

\begin{proposition}\label{theorem-MMS}
Let $n=3$. For any initial data
\begin{equation}
Q_0\in \overline{D(\E)}:=\overline{\{Q\in L^2(\TT^3;\QQ)\mid \mathcal{E}(Q)<\infty\}}^{L^2(\TT^3)},
\end{equation}
Let $Q(t)=\displaystyle\lim_{k\rightarrow+\infty}J_{t/k}^k(Q_0)$ with
$J$ being the resolvent
$$
  X\in J_\tau(Y) \Longleftrightarrow X\in\mbox{argmin}\Big\{F(\cdot)+\frac{1}{2\tau}\|Y-\cdot\|^2
  \Big\}.
$$
Then we have
\begin{enumerate}
\item Variational inequality: $Q$ is the unique solution to the
evolution variational inequality
\begin{equation}\label{variational-inequality solution}
\begin{split}
\dfrac{d}{dt}\|Q(t,\cdot)-P\|^2_{L^2(\TT^3)}&-\alpha\|Q(t,\cdot)-P\|^2_{L^2(\TT^3)}+\E(Q(t,\cdot))\leq\E(P),\\
  &\quad \mbox{for a.e. } t>0 \mbox{ and } P\in D(\E),
\end{split}
\end{equation}
among all locally absolutely continuous curves such that
$Q(t,\cdot)\rightarrow Q_0$ as $t \downarrow 0^+$.
\item Regularizing effect: $Q$ is locally Lipschitz, and
$Q(t,\cdot)\in D(\E)$
for all $t>0$. In
particular, $Q$ is physical in the sense that
\begin{equation}
 -\dfrac13<\lambda_i(Q(t, x))<+\dfrac23, \quad \forall t>0, \; a.e.\; x\in\TT^3
\end{equation}
\end{enumerate}
\end{proposition}

To proceed, note that $\E$ is $-2\alpha$ convex in $L^2(\TT^3;\mathcal{Q})$, hence by Remark \ref{remark-equivalence}
we know that $Q(t,\cdot)$ constructed in Proposition \ref{theorem-MMS} is the unique solution to the gradient flow \eqref{grad flow E}.
The following lemma computes explicitly the sub-differential of the free energy $\E$.

\begin{lemma}\label{lemma-subgradient}
For any $Q\in D(\partial\E)$ and $1\leq i, j\leq 3$, we have
\begin{align*}
-\partial\E(Q)_{ij}&=2L_1\Delta{Q}_{ij}+2(L_2+L_3)\p_{kj}Q_{ik}-\dfrac{2(L_2+L_3)}{3}\p_{k\ell}Q_{k\ell}\delta_{ij}\\
&\qquad-\psi'(Q)_{ij}+\dfrac{\mathrm{tr}(\psi '(Q))}{3}\delta_{ij}+2\alpha{Q}_{ij}.
\end{align*}
\end{lemma}
\begin{proof}

To begin with, it is immediate to derive
\begin{equation}
\partial\G(Q)_{ij} =-2L_1 \Delta Q_{ij}
-2(L_2+L_3)\partial_k\partial_jQ_{ik}+\frac23 (L_2+L_3)\partial_{\ell}\partial_kQ_{k\ell} \delta_{ij}.
\end{equation}

Next we need to verify that
$$
\partial\mathcal{B}(Q)=\Big\{\psi'(Q)-\frac13\mathrm{tr}(\psi '(Q))\mathbb{I}_3\Big\}.
$$

\noindent{\em Case 1: $Q$ is strictly physical}.

\medskip

Let $R\in C_c^{\infty}(\TT^3, \mathcal{Q})$.
Then by convexity and smoothness of $\psi$, any element $\xi\in\partial\mathcal{B}(Q)$ satisfies
\begin{align*}
\int_{\TT^3}\big[\psi '(Q): R\big]\,d{x}&=\lim_{\eps\rightarrow 0^+}\frac{\mathcal{B}(Q+\eps R)-\mathcal{B}(Q)}{\eps}\geq \langle\xi, R\rangle_{L^2(\TT^3)},\\
-\int_{\TT^3}\big[\psi '(Q): R\big]\,d{x}&=\lim_{\eps\rightarrow
0^+}\frac{\mathcal{B}(Q-\eps R)-\mathcal{B}(Q)}{\eps}\geq -\langle\xi,
R\rangle_{L^2(\TT^3)},
\end{align*}
which indicates $\langle\xi, R\rangle_{L^2(\TT^3)}= \langle
\psi '(Q), R\rangle_{L^2(\TT^3)}$. By density,
\begin{equation}
\xi=\psi '(Q)-\mathrm{tr}(\psi '(Q))\mathbb{I}_3/3
\end{equation}
 as an
element in the Hilbert space $L^2(\TT^n;\mathcal{Q})$. Hence
$\partial\mathcal{B}(Q)$ $=\{\psi'(Q)-\mathrm{tr}(\psi'(Q))\mathbb{I}_3/3\}$ for any uniformly physical $Q$.

\medskip

\noindent{\em Case 2: $Q$ is not strictly physical}.

\medskip

We define
\begin{equation}\label{def-rho}
\rho(P) := \min_{1\leq i\leq
3}\Big\{\lambda_i(P)+\frac13,\,\frac23-\lambda_i(P)\Big\},\qquad\forall\,
P\in\mathcal{Q},
\end{equation}
and $A_\eta := \{x\in \TT^3:\rho(Q(x))<\eta\}$ for arbitrarily
small $\eta>0$. Since $Q\in D(\partial\E)\subset D(\E)$, we have
$\psi(Q)<+\infty$, thus $|A_\eta|\rightarrow 0$ as $\eta\rightarrow
0^+$. Let
\begin{equation*}
T_\eta(Q) := \{R\in L^\infty(\TT^3; \mathcal{Q}):\ R\equiv 0 \mbox{ on }
A_\eta \}.
\end{equation*}

Let us consider any fixed $\eta>0$, and $R\in T_\eta(Q)$. Then it is
easy to check that $\rho(Q)\ge \eta$ outside $A_\eta$, and for all
sufficiently small $\eps$ we have $\rho(Q\pm\eps R)\ge \eta/2$
outside $A_\eta$, hence
$$
\mathcal{B}(Q+\eps R) =\int_{\TT^3\setminus A_\eta }  \psi (Q+\eps
R)\,\ud{x}+\int_{A_\eta}\psi (Q)\,\ud{x}<+\infty.
$$
This together with the covexity of $\psi$ implies that for any sufficiently small $\eps>0$
\begin{align*}
&\eps\int_{\TT^3\setminus A_\eta}\big[\psi '(Q+\eps R): R\big]\,\ud{x}
\geq\mathcal{B}(Q+\eps R)-\mathcal{B}(Q) \\
=&\int_{\TT^3\setminus
A_\eta}\big[\psi (Q+\eps
R)-\psi (Q)\big]\,\ud{x}
\geq\eps\int_{\TT^3\setminus A_\eta}\big[\psi '(Q): R\big]\,\ud{x},
\end{align*}
and dividing by $\eps$ gives
\begin{align}
\liminf_{\eps\rightarrow 0^+}\frac{\mathcal{B}(Q+\eps R)-\mathcal{B}(Q)}{\eps}&\geq\int_{\TT^3\setminus A_\eta }\big[\psi '(Q): R\big]\,\ud{x},\label{inf-1}\\
\limsup_{\eps\rightarrow 0^+}\frac{\mathcal{B}(Q+\eps R)-\mathcal{B}(Q)}{\eps}
&\leq \limsup_{\eps\rightarrow 0^+}\int_{\TT^3\setminus A_\eta }\big[\psi '(Q+\eps R): R\big]\,\ud{x}.\label{sup-1}
\end{align}
Meanwhile, note that
$$
\psi '\in L^\infty(\TT^3\setminus A_\eta ), \;\;
\lim_{\eps\rightarrow{0^+}}\int_{\TT^3\setminus A_\eta}
\big[\psi '(Q+\eps R): R\big]\,\ud{x}=\int_{\TT^3\setminus A_\eta }
\big[\psi '(Q): R\big]\,\ud{x}=\int_{\TT^3}\big[\psi '(Q): R\big]\,\ud{x}
$$
due to the fact that $R\equiv 0$ on $A_\eta$. This together with
\eqref{inf-1}, \eqref{sup-1} yields
\begin{equation*}
\lim_{\eps\rightarrow 0^+}\frac{\mathcal{B}(Q+\eps R)-\mathcal{B}(Q)}{\eps}=
\int_{\TT^3}\big[\psi '(Q): R\big]\,\ud{x}.
\end{equation*}
Further, as discussed in case 1, for any
$R\in \bigcup_{\eta>0} T_\eta(Q)$
we have
$$
\langle\xi, R\rangle_{L^2(\TT^3)}= \langle\psi '(Q), R\rangle_{L^2(\TT^3)}.
$$
By density, it follows
$\xi=\psi '(Q)-\mathrm{tr}(\psi'(Q))\mathbb{I}_3/3$ as elements of the Hilbert space $L^2(\TT^3; \mathcal{Q})$.
Hence $\partial\mathcal{B}(Q)=\{\psi'(Q)-\mathrm{tr}(\psi '(Q))\mathbb{I}_3/3\}$ even if $Q\subset D(\partial\E)$ is not strictly physical.

By Lemma \ref{lemma-lsc}, $\G$ and
$\mathcal{B}$ are proper, convex and lower semicontinuous, and the
intersection $D(\G)\cap \text{int} D(\mathcal{B})$ is non empty, we get
$\partial(\G+\mathcal{B})(Q)=\partial\G(Q)+\partial\mathcal{B}(Q)$ by
\cite[Theorem~2.10]{B10}. Since the last term
$-\alpha\|Q\|_{L^2(\TT^3)}^2$ is a $C^1$ perturbation of the energy,
we infer
\begin{equation}
 \partial\E(Q)=\partial(\G+\mathcal{B})(Q)-2\alpha{Q}=\partial\G(Q)+\partial\mathcal{B}(Q)-2\alpha{Q},
\end{equation}
 which concludes the proof.
\end{proof}

Proposition \ref{theorem-MMS} leads to  higher regularity of the solution $Q$. Since
the energy $\E $ is $-2\alpha$-convex, and since the solution $Q(t,\cdot)$ satisfies
$Q(t,\cdot)\in D(\E)$ for all time $t>0$, we conclude that for any $t\geq t_0>0$ the function $Q(t,\cdot)$
is the gradient flow of $\E$ in $L^2(\TT^n;\mathcal{Q})$ with initial datum
$Q(t_0, \cdot)\in D(\E)$. Thus we can apply
\cite[Theorem 2.4.15]{AGS08} to obtain
\begin{proposition}
\label{prop-energy-decay}
Let $Q(t, \cdot)$ be the solution given by Theorem \ref{main-theorem-1}. Then
the map
$$
  t\mapsto e^{-2\alpha t} \|\partial\E(Q(t,\cdot))\|_{L^2(\TT^3)}
$$
is nonincreasing, right continuous on $[t_0,+\infty)$ for all
$t_0>0$.
\end{proposition}

Finally, combining \cite[Corollary~2.4.11]{AGS08} and
\cite[Theorem~2.3.3]{AGS08} we obtain the energy identity \eqref{energy identity}

\begin{proposition}\label{prop-L2L2}
The solution $Q(t, \cdot)$ given by Theorem \ref{main-theorem-1} satisfies the
energy equality
\begin{equation*}
\int_{t_0}^T\frac12\big(\|\partial_tQ(t,\cdot)\|_{L^2(\TT^3)}^2+\frac12\|\partial\E(Q(t,\cdot))\|^2_{L^2(\TT^3)}\big)\,\ud{t}
=\E(Q(t_0))-\E(Q(T))
\end{equation*}
for all $0<t_0<T<+\infty$.
\end{proposition}
In conclusion, the proof of Theorem \ref{main-theorem-1} is complete.
\begin{remark}
The nonincreasing property of the energy term $e^{-2\alpha t} \|\partial\E(Q(t,\cdot))\|_{L^2(\TT^3)}$ in Proposition \ref{prop-energy-decay} will play an essential
role in the proof of Theorem \ref{main-theorem-2}, which is the main reason that the Ambrosio-Gigli-Savare's gradient flow theory in \cite{AGS08}
is adopted in this section other than the classical Brezis-Pazy's theory.
\end{remark}


\section{Proof of Theorem \ref{main-theorem-2}: higher regularity of solutions}
This section is devoted to the proof of Theorem \ref{main-theorem-2}.
Since there is only minor modification of arguments between $\TT^2$ and $\TT^3$, we only discuss the case in $\TT^3$.

It is noted that the maximum principle argument utilized in \cite{W12} fails due to the presence of the anisotropic terms, hence the solution $Q(t, \cdot)$ is not ensured to
stay detached from the physical boundary $\p \mathcal{Q}$ at any positive time $t$. To achieve the proof, we have to put together several results in \cite{AGS08, SS04, S11}, and to  make a full exploitation of the gradient flow structure in \eqref{grad flow E}.

Our main strategy is as follows. First of all,
to avoid the singular feature of $\partial\E(Q)$ in \eqref{grad flow E}, we shall consider a sequence of smooth gradient flows \eqref{grad flow N} that are generated by an approximation sequence $\{\E_n\}$ defined in \eqref{energy-sequence} of the free energy $\E$. Secondly, we will prove $\Gamma-$ convergence of $\{\E_n\}$ to $\E$ in Proposition \ref{proposition-gamma-convergence}, which together with energy dissipative equality achieved in Proposition \ref{prop-L2L2} we can show the ``convergence" of the gradient flow sequence \eqref{grad flow N} to \eqref{grad flow E}. Next, we will show in Proposition \ref{lemma-H2-2} that the solution sequence $\{Q_n\}$ to the gradient flow sequence \eqref{grad flow N} is in $H^2(\TT^3; \mathcal{Q})$ space, and give a corresponding estimate of the $H^2$ bound of $Q$ in terms of
$\|\partial\E(Q)\|$, which together with Proposition \ref{prop-energy-decay} leads to the uniform-in-time bound for $\|\partial\E(Q)\|$. Finally, we make
use of the convexity of $\psi $ and Sobolev interpolation inequalities to derive strict physicality of the solution at all large times.

Here and after, we denote $\{\psi_n\}$ the sequence of functions
that is used in \cite{W12} to approximate the Ball-Majumdar bulk
potential $\psi $:  first, we introduce the Moreau-Yosida approximations
\begin{equation}\label{Moreau-Yosida}
\tilde{\psi}_n(Q) := \inf_{A\in \mathcal{Q}}\{ n|A-Q|^2+\psi (A)\}, \qquad Q\in\mathcal{Q}.
\end{equation}
Then using a smooth regularization we define
\begin{equation}\label{mollification}
\psi_n(Q)=n^5\int_{\mathcal{Q}}\tilde{\psi}_n\big(n(Q-R)\big)\phi(R)\,dR, \qquad Q\in\mathcal{Q}.
\end{equation}
Here $\phi\in C_c^{\infty}( \mathcal{Q} , \overline{R^+})$ is of unit mass.
Let us recall \cite[Proposition~3.1]{W12} for each $n\geq 1$, we have

\bigskip

\noindent$(\textbf{M0})$ $\psi_n$ is an isotropic function of $Q$.

\smallskip

\noindent$(\textbf{M1})$ $\psi_n$ is both smooth and convex in
$\mathcal{Q}$.

\smallskip

\noindent$(\textbf{M2})$ $\psi_n$ is bounded from below, i.e.,
$-4\pi^2|\TT^3|/e\leq\psi_n(R)$, $\forall R\in\mathcal{Q}$,
$\forall n\geq 1$.

\smallskip

\noindent$(\textbf{M3})$ $\psi_n\leq\psi_{n+1}\leq\psi $ on
$\mathcal{Q}$ for $n\geq 1$.

\smallskip

\noindent$(\textbf{M4})$ $\psi_n\rightarrow\psi $ in
$L_{loc}^{\infty}\big(D(\psi )\big)$ as $n\rightarrow\infty$, and
$\psi_n$ is uniformly divergent on $\mathcal{Q}\setminus
D(\psi )$

\smallskip

\noindent$(\textbf{M5})$
$\frac{\partial\psi_n}{\partial{Q}}\rightarrow\frac{\partial\psi}{\partial{Q}}$
in $L_{loc}^{\infty}\big(D(\psi )\big)$ as $n\rightarrow\infty$.

\smallskip

\noindent$(\textbf{M6})$ There exist constants $\lambda_n, \Lambda_n>0$ that
may depend on $n$, such that
$$
  \lambda_n|R|-\Lambda_n\leq\Big|\psi_n'(R)-\frac13\mathrm{tr}(\psi_n'(R))\mathbb{I}_3\Big|\leq \lambda_n|R|+\Lambda_n, \quad\forall R\in\mathcal{Q}
$$

Besides the aforementioned properties $(\textbf{M0})$ to $(\textbf{M6})$, we need to further derive the following finer estimate
of the sequence $\{\psi_n\}$, in order to prove the later Proposition \ref{proposition-gamma-convergence} and Lemma \ref{lemma-grad}.
\begin{lemma}\label{lemma-approximation-sequence}
For any $n\in\mathbb{N}$, there exists a generic constant $C_n>0$ such that
\begin{align}\label{lower-bound}
\psi_n(Q)&\geq C_n|Q|^2, \qquad\mbox{outside a fixed compact
subset in } \mathcal{Q}.
\end{align}
Moreover, $C_n\rightarrow+\infty$ as $n\rightarrow+\infty$.
\end{lemma}
\begin{proof}

Since the null matrix $0\in D(\psi )$, taking $A=0$ in \eqref{Moreau-Yosida} we get
the upper bound
$$
 \tilde{\psi}_n(Q)\leq n|Q|^2+\psi (0).
$$
On the other hand, since such infimum in \eqref{Moreau-Yosida} is finite, there exists a
minimizing sequence $A_m\subset D(\psi )$ such that
\begin{align}
&\tilde{\psi}_n(Q)=\lim_{m\rightarrow+\infty} (n\,|A_m-Q|^2+\psi (A_m))\nonumber\\
\geq &\lim_{m\rightarrow+\infty}
n|A_m-Q|^2+\inf\psi  \geq n\, \mbox{dist}(Q,D(\psi ))^2+\inf\psi.
\end{align}

By the triangle inequality, we have
$$
  |Q|\leq \mbox{dist}(Q,D(\psi ))+\mbox{dist}(0,D(\psi )),
$$
and since $0\in D(\psi)$, we can further get
$$
\mbox{dist}(Q,D(\psi )) \geq |Q|-\mbox{diam}(D(\psi )).
$$
As a consequence, we see
$$
\tilde{\psi}_n(Q)\geq n|Q|^2-2n\,\mbox{diam}(D(\psi ))
|Q|+\mbox{diam}(D(\psi ))^2+\inf\psi \geq
n|Q|\Big(|Q|-2\mbox{diam}(D(\psi ))\Big).
$$
Hence we arrive at the uniform quadratic estimate
\begin{equation}
\frac{n}{2}|Q|^2 \le  \tilde{\psi}_n(Q)
\end{equation}
in $\{|Q| > 4\mbox{diam}(D(\psi ))\}$. Then following the mollification of
$\tilde{\psi}_n$ in \eqref{mollification} one can get \eqref{lower-bound}.
\end{proof}


We introduce the energy sequence $\E_n: \mathcal{Q}\rightarrow\RR\cup\{+\infty\}$
\begin{equation}\label{energy-sequence}
\E_n(Q)=\begin{cases}
\G(Q)+\int_{\TT^3}\psi_n(Q)\,\ud{x}-\alpha\|Q\|_{L^2(\TT^3)}^2, \quad\mbox{if } Q\in H^1(\TT^3)\\
+\infty, \qquad\qquad\qquad\qquad\qquad\qquad\qquad\mbox{otherwise},
\end{cases}
\end{equation}
and establish a $\Gamma$-convergence result.

\begin{proposition}\label{proposition-gamma-convergence}
The sequence of energies $\{\E_n\}$ $\Gamma$-converges to $\E$.
\end{proposition}
\begin{proof}
We first show compactness. Let us assume
$\displaystyle\liminf_{n\rightarrow+\infty}\E_n(Q_n)<+\infty $. Upon subsequence,
we may assume
\begin{equation}\label{contradiction}
\liminf_{n\rightarrow+\infty}\E_n(Q_n)=\lim_{n\rightarrow+\infty}\E_n(Q_n)<+\infty
,\qquad \displaystyle\sup_{n\in\mathbb{N}}\E_n(Q_n)<+\infty.
\end{equation}
We need first to ensure the existence of a strong limit $Q$. We claim that $Q_n$ is uniformly bounded in $L^2(\TT^3)$. Otherwise there exists a subsequence $\{Q_{n_k}\}$, such that $\|Q_{n_k}\|_{L^2(\TT^3)}\rightarrow +\infty$, it then follows directly from
Lemma \ref{lemma-lsc} and \eqref{lower-bound} that $\E_{n_k}(Q_{n_k})\rightarrow+\infty$, which is in contradiction with the assumption \eqref{contradiction}.

Further, note that
\begin{align*}
&(L_1-3|L_2+L_3|)\sup_n \|\nabla{Q}_n\|_{L^2(\TT^3)}^2\\
&\leq\sup_n\G(Q_n)\leq \sup_n\E_n(Q_n)+\alpha\sup_n\|Q_n\|_{L^2(\TT^3)}^2
+\inf_{n,P}\int_{\TT^3}|\psi_n(P)|\,\ud{x} <+\infty.
\end{align*}
Thus $Q_n$ is uniformly bounded in $H^1(\TT^3)$, and
$Q_n\rightarrow Q$ (up to a subsequence) strongly in $L^2(\TT^3)$.

\smallskip

Next we show $\Gamma-\limsup$ inequality. That is, for any $Q\in\mathcal{Q}$
there exists a recovery sequence $Q_n$ such that
\begin{equation}\label{Gamma-limsup}
\limsup_{n\rightarrow+\infty}\E_n(Q_n)\leq \E(Q).
\end{equation}
Without loss of generality we assume $\E(Q)<+\infty$. Taking
$Q_n=Q$, for every $n\geq 1$ we get from $(\textbf{M3})$ that
$$
\G(Q_n)-\alpha\|Q_n\|_{L^2(\TT^3)}^2\equiv\G(Q)-\alpha\|Q\|_{L^2(\TT^3)}^2,\quad\int_{\TT^3}
\psi_n(Q_n)\,\ud x \leq \int_{\TT^3} \psi(Q)\,\ud x, \quad\forall
n\geq 1.
$$
In all, \eqref{Gamma-limsup} is verified.

\smallskip

We proceed to show $\Gamma-\liminf$ inequality. That is, for any
sequence $Q_n\rightarrow Q$ strongly in $L^2(\TT^3)$, it holds
\begin{equation}\label{Gamma-liminf}
\liminf_{n\rightarrow+\infty}\E_n(Q_n)\geq\E(Q).
\end{equation}
Without loss of generality we assume
$\displaystyle\liminf_{n\rightarrow+\infty}\E_n(Q_n)<+\infty$. Again
upon subsequence, we may assume
$$
\liminf_{n\rightarrow+\infty}\E_n(Q_n)=\lim_{n\rightarrow+\infty}\E_n(Q_n)<+\infty
,\qquad \displaystyle\sup_{n\in\mathbb{N}}\E_n(Q_n)<+\infty.
$$
As discussed earlier we have $Q_n\rightarrow Q$ strongly in
$L^2(\TT^3)$, and upon further extracting a subsequence we may
assume $Q_n\rightarrow Q$ a.e. in $\TT^3$. Hence together with the lower semicontinuity of $\G$ achieved in
Lemma \ref{lemma-lsc} it yields
\begin{equation}
\lim_{n\rightarrow+\infty}\|Q_n(\cdot)\|_{L^2(\TT^3)}^2=\|Q(\cdot)\|_{L^2(\TT^3)}^2,\qquad\liminf_{n\rightarrow+\infty}\G(Q_n)
\geq\G(Q).
\end{equation}
It remains to prove
\begin{equation}\label{Gamma-inf-1}
\liminf_{n\rightarrow+\infty}\int_{\TT^3}\psi_n(Q_n)\,\ud x\geq
\int_{\TT^3}\psi (Q)\,\ud x.
\end{equation}
To proceed, we denote
$$
  D:=\big\{x\in\TT^3: Q(x)\in D(\psi )\big\},
$$
and we distinguish between two cases.

\bigskip

\noindent{\em Case 1: $\mathcal{B}(Q)<+\infty$.} For any sufficiently small $\eps>0$ let us define
\begin{equation}\label{Omega-eps}
  \TT^3_\eps=D\cap\left\{x\in\TT^3:-\frac13+\eps\leq\lambda_i(Q(x))\leq\frac23-\eps,\;1\leq
  i\leq 3\right\}.
\end{equation}
Since $\big|\TT^3\setminus{D}\big|=0$ we get
\begin{align}
\int_{\TT^3}\big[\psi_n(Q_n)-\psi (Q)\big]\,\ud{x}
&=\int_{D}\big[\psi_n(Q_n)-\psi (Q)\big]\,\ud{x}\non\\
&=\int_{D\setminus\TT^3_\varepsilon}\big[\psi_n(Q_n)-\psi (Q)]\,\ud{x}
+\int_{\TT^3_\eps}[\psi_n(Q_n)-\psi (Q)\big]\,\ud{x}.
\end{align}
On the one hand, for any $n\geq 1$,
\begin{align*}
\int_{D\setminus\TT^3_\varepsilon}\psi_n(Q_n)\,\ud{x}&\geq
\lim_{\eps\to 0}(\inf_{n,P}\psi_n(P))\big|D\setminus \TT^3_\varepsilon
\big|=0,
\end{align*}
and since $\psi \in L^1(\TT^3)$, as $|D\setminus\TT^3_\varepsilon|\rightarrow 0$, we have
\begin{align*}
\lim_{\eps\to 0}\int_{D\setminus\TT^3_\varepsilon}\psi (Q)\,\ud{x}=0.\end{align*}
Hence for all $\delta>0$, there exists
$\eps_0=\eps_0(\delta)>0$, such that
\begin{equation}\label{delta-1}
\int_{D\setminus\TT^3_\varepsilon}\big[\psi_n(Q_n)-\psi (Q)\big]\,\ud{x}>-\dfrac{\delta}{2},
\quad \forall n\geq 1 \quad\mbox{whenever }\, \eps\leq\eps_0.
\end{equation}

On the other hand, it follows from
$(\textbf{M5})$ that $\psi_n(Q(x))\rightarrow\psi (Q(x))$,
$\psi_n'(Q(x))\rightarrow\psi '(Q(x))$ as $n\rightarrow+\infty$ in $L^\infty(\TT^3_\eps)$.
Thus it implies for any fixed $\eps \leq \eps_0$ it
holds
$$
 \lim_{n\rightarrow\infty}\int_{\TT^3_\eps}[\psi_n(Q_n)-\psi (Q)\big]\,\ud{x}=0.
$$
Specifically, for $\eps=\eps_0$ there exists $N_0\in\mathbb{N}$,
such that
\begin{equation}\label{delta-2}
\int_{\TT^3_{\eps_0}}[\psi_n(Q_n)-\psi (Q)\big]\,\ud{x}>-\dfrac{\delta}{2},\quad\forall
n\geq N_0
\end{equation}
Combining \eqref{delta-1} and \eqref{delta-2} we finish the proof of
\eqref{Gamma-inf-1} by the arbitrariness of $\delta$.

\bigskip

\noindent{\em Case 2: $\mathcal{B}(Q)=+\infty$.} In this case it suffices
to check
\begin{equation}\label{Gamma-inf-2}
\liminf_{n\rightarrow\infty}\int_{\TT^3}\psi_n(Q_n)\,\ud{x}=+\infty.
\end{equation}
If $|\TT^3\setminus{D}|>0$, since $Q_n\rightarrow Q$ a.e. it
follows from Egorov's theorem that there exists a set $F\subset(\TT^3\setminus{D})$,
$|F|>0$, such that $Q_n\rightarrow Q$ uniformly on $F$. Note that
$Q(x)\in\mathcal{Q}\setminus D(\psi )$, $\forall x\in F$.
Hence the uniform convergence of $Q_n$ to $Q$ on $F$ implies there
exists a sequence $\eps_n\searrow 0^+$, such that
$$
  \lambda_i(Q_n(x))\leq-\frac13+\eps_n \;\mbox{or }\;
  \lambda_i(Q_n(x))\geq\frac23-\eps_n, \quad\forall\,1\leq i\leq 3.
$$
Then Fatou's lemma and $(\textbf{M4})$ yields
$$
  \liminf_{n\rightarrow+\infty} \int_F \psi_n(Q_n)\,\ud{x} \geq  \int_F \liminf_{n\rightarrow+\infty}\psi_n(Q_n)\,\ud{x}=+\infty.
$$
Therefore, \eqref{Gamma-inf-2} is verified
in this case.

\smallskip

Alternatively if $|\TT^3\setminus{D}|=0$, then using similar
argument as in \emph{Case 1} we have
$$
 \int_{\TT^3}\psi_n(Q_n)\,\ud{x}=\int_D\psi_n(Q_n)\,\ud{x}=\int_{D\setminus\TT^3_{\eps}}\psi_n(Q_n)\,\ud{x}+\int_{\TT^3_{\eps}}\psi_n(Q_n)\,\ud{x},
$$
where $\TT^3_{\eps}$ is given in \eqref{Omega-eps}. On one hand,
\begin{equation}\label{Gamma-inf-2-1}
 \int_{D\setminus\TT^3_{\eps}} \psi_n(Q_n)\,\ud{x}\geq (\inf_{n,P}\psi_n(P))\big|\TT^3\setminus\TT^3_{\eps}\big|
 \geq-\frac{4\pi|\TT^3|^2}{e},\quad\forall\eps>0.
\end{equation}
On the other hand, since $\mathcal{B}(Q)=+\infty$ and
$\TT^3_{\eps}\nearrow D$, we infer that $\forall M>0$, there exists
$\eps_0>0$, such that
$$
  \int_{\TT^3_{\eps}}\psi (Q)\,\ud{x}>M+1+\frac{4\pi|\TT^3|^2}{e},\quad\forall\eps\leq\eps_0.
$$
Meanwhile, $(\textbf{M5})$ indicates that
$\psi_n(Q(x))\rightarrow\psi (Q(x))$,
$\psi_n'(Q(x))\rightarrow\psi '(Q(x))$ in
$L^\infty(\TT^3_{\eps_0})$, which gives
\begin{equation*}
\int_{\TT^3_{\eps_0}}\psi_n(Q_n)\,\ud{x}\rightarrow
\int_{\TT^3_{\eps_0}}\psi (Q)\,\ud{x}.
\end{equation*}
Thus there exists $N\in\mathbb{N}$, such that
\begin{equation}\label{Gamma-inf-2-2}
\int_{\TT^3_{\eps_0}}\psi_n(Q_n)\,\ud{x}>\int_{\TT^3_{\eps_0}}\psi (Q)\,\ud{x}-1>M+\frac{4\pi|\TT^3|^2}{e},
\quad\forall n\geq N.
\end{equation}
Putting together \eqref{Gamma-inf-2-1} and \eqref{Gamma-inf-2-2}, we
conclude that \eqref{Gamma-inf-2} is valid. Hence the proof is
complete.
\end{proof}

Since in Theorem \ref{main-theorem-1} we allow the initial data
$Q_0\in \overline{D(\E)}$, one may not expect any further regularity
at time $t=0$. Thus we only aim to prove the regularity of $Q(t,\cdot)$ for
positive times $t>0$.

\smallskip

Let $t_0>0$ be arbitrarily given. Consider the gradient flow
sequence
\begin{align}\label{grad flow N}
\begin{cases}
\partial_tQ_n\in -\partial\E_n(Q_n),\quad t\geq t_0,\\
Q_n(t_0, x)=Q(t_0, x)
\end{cases}
\end{align}
subject to periodic boundary conditions \eqref{BC-periodic}.
Here we let time start from $t_0$ only as a matter of
convenience, to avoid an (unnecessary) time shifting. Note that the
energies $\E_n$ are also $-2\alpha$-convex, proper, lower semicontinuous,
and uniformly bounded from below. Thus we can apply
\cite[Theorem~4.0.4]{AGS08} to get the same conclusions as in
Theorem \ref{main-theorem-1}, Proposition \ref{prop-energy-decay}
and Proposition \ref{prop-L2L2}.

\smallskip

As a consequence, we manage to prove

\begin{lemma}\label{lemma-grad} Let $Q_n$ be the solution of
\eqref{grad flow N}. Then for every $t_0\in (0, T)$
\begin{equation*}
\|\partial\E_n(Q_n(t,\cdot))\|_{L^2(\TT^3)}\rightarrow
\|\partial\E(Q(t,\cdot))\|_{L^2(\TT^3)} ,\quad
\|\partial_tQ_n(t,\cdot)\|_{L^2(\TT^3)}\rightarrow
\|\partial_tQ(t,\cdot)\|_{L^2(\TT^3)} \;\text{ in } L^2(t_0,T),
\end{equation*}
up to a subsequence.
\end{lemma}
\begin{proof}
By Proposition \ref{proposition-gamma-convergence}, we have
$\E_n\overset\Gamma\rightarrow\E$. We aim to show that we are under
the hypotheses of Proposition \ref{theorem-serfaty}, where the general sense of convergence is considered to be the strong convergence
in $L^2(\TT^3; \mathcal{Q})$.
Let us first check that the conditions \eqref{Serfaty-condition1-1}, \eqref{Serfaty-condition1-2} are valid for the solutions $Q_n$ of \eqref{grad flow
N}.

By Proposition \ref{prop-L2L2} and $(\textbf{M3})$ we have
$$
\E_n(Q_n)(t)\leq\E_n(Q(t_0))\leq\E(Q(t_0)),\quad \forall\, t>t_0,
n\in\mathbb{N}.
$$
Hence following a similar argument in the proof of compactness part in Proposition
\ref{proposition-gamma-convergence} we get
\begin{equation}
  \{Q_n\}\;\;\mbox{is uniformly bounded in } L^\infty(t_0, T; H^1(\TT^3)).
\end{equation}

Next, by Proposition \ref{prop-L2L2} and $(\textbf{M3})$, the energy equality
\begin{align}\label{energy-equality}
\int_{t_0}^T\big(\|\partial_tQ_n(t,\cdot)\|^2+\|\partial\E_n(Q_n(t,\cdot))\|^2\big)\,\ud{t}&=
2\E_n(Q(t_0))-2\E_n(Q_n(T))\non\\
&\leq 2\E(Q(t_0))-2\inf\E_n
\end{align}
holds for all $t_0<T<+\infty$. Hence
\begin{equation}
\{\partial_tQ_n\}\;\mbox{is uniformly bounded in}\; L^2(t_0, T;
L^2(\TT^3))
\end{equation}
and we can apply the Aubin-Lions lemma to yield
\begin{equation}\label{convergence-cite}
Q_n \rightarrow \tilde{Q}\;\;\mbox{strongly in } C\big([t_0, T];
L^2(\TT^3)\big),
\end{equation}
and \eqref{energy-equality} implies that
\begin{equation}
\partial_tQ_n\rightarrow \partial_t\tilde{Q}\;\;\mbox{weakly in }
L^2(0, T; L^2(\TT^3)).
\end{equation}
Hence the condition \eqref{Serfaty-condition1-1} is satisfied by taking $f=0$. Since each
$\E_n$ is $-2\alpha$ convex, by \cite[Proposition 13]{O05} we know that the condition
\eqref{Serfaty-condition1-2} is also satisfied by taking $C=0$.

\smallskip

Finally, note that the initial data are ``well-prepared" in the
sense of $\E_n(Q(t_0))\rightarrow \E(Q(t_0))$, thus all the
assumptions in Proposition \ref{theorem-serfaty} are satisfied, which in turn
gives that $\tilde{Q}$ is a solution of the gradient flow
$$
\begin{cases}
\partial_t\tilde{Q}\in-\partial\E(\tilde{Q}),\\
\tilde{Q}(t_0)=Q(t_0).
\end{cases}
$$
Since $Q$ is already a solution, and by Theorem \ref{main-theorem-1}
we know that such a solution is unique, we infer $\tilde{Q}=Q$, and
the proof is complete.

\end{proof}
%


Now we turn to  the  $H^2$-regularity of the sequence $\{Q_n\}$.  The
following technical implies the coercivity of $\partial\G$:
\begin{lemma}\label{lemma-H2}
For any $P\in H^2(\TT^3;\mathcal{Q})$ the operator $\partial\G(P)$ satisfies
\begin{equation*}
2(L_1-|L_2+L_3|) \|\Delta P\|^2_{L^2(\TT^3)} \leq  \langle
-\partial \G(P), \Delta P\rangle_{L^2}\leq
2(L_1+|L_2+L_3|)\|\Delta{P}\|_{L^2(\TT^3)}^2.
\end{equation*}
\end{lemma}
\begin{proof}
Direct
computations yield
\begin{align*}
\langle -\partial\G(P), \Delta P\rangle_{L^2}
&=\int_{\TT^3}\bigg[2L_1\Delta{P}_{ij}+2(L_2+L_3)\partial_k\partial_jP_{ik}-\frac{2}{3}(L_2+L_3)\partial_\ell\partial_kP_{k\ell}\delta_{ij}
\bigg]\Delta{P}_{ij}\,\ud{x}\\
&=2L_1\|\Delta P\|^2_{L^2(\TT^3)}+2(L_2+L_3)\int_{\TT^3}\partial_j\partial_kP_{ik}\Delta P_{ij}\,\ud{x},
\end{align*}
where the last term of R.H.S. can be treated by integration by parts and the H\"{o}lder's inequality:
\begin{align*}
2(L_2+L_3)\int_{\TT^3}\partial_k\partial_jP_{ik}\Delta P_{ij}\,\ud{x}&=2(L_2+L_3)\int_{\TT^3}\partial_{\ell}\partial_jP_{ik}\partial_\ell\partial_kP_{ij}\,\ud{x}\\
&\leq 2|L_2+L_3|\int_{\TT^3}\sqrt{\displaystyle\sum_{i,j,k,\ell}(\partial_\ell\partial_jP_{ik})^2}\sqrt{\displaystyle\sum_{i,j,k,\ell}(\partial_\ell\partial_kP_{ij})^2}\,\ud{x}\\
&=2|L_2+L_3|\int_{\TT^3}|\nabla^2P|^2\,\ud{x}=2(|L_2+L_3|) \|\Delta P\|^2_{L^2(\TT^3)}.
\end{align*}
\end{proof}

Next, we recall the notion of angles between two elements in a Hilbert space $(H, \langle, \rangle)$. Given two nonzero elements
$u, v\in H$, the angle between $u$ and $v$ is the unique value
$$
  \angle (u, v) := \arccos\frac{\langle u, v\rangle}{\|u\|\|v\|}\in [0,\pi], \qquad\text{where}\quad\|u\|_{H}=\sqrt{\langle u, u\rangle}.
$$
Then the following triangle inequality is valid.
\begin{lemma}\label{lemma-triangle-inequality}
For any three unit vectors $\vec{u}, \vec{v}, \vec{w}\in\mathbb{S}^2$, it holds
\begin{equation}\label{triangle-inequality-2}
\angle (\vec{u}, \vec{v})\leq \angle (\vec{u}, \vec{w})+\angle (\vec{v}, \vec{w}),
\end{equation}
where $\angle (\vec{u}, \vec{v})\in [0, \pi]$ stands for the angle between the two vectors $\vec{u}$ and $\vec{v}$.
\end{lemma}
\begin{proof}
By rotation invariance, we assume the unit vectors $\vec{u}$, $\vec{v}$ lie on the $xy$-plane. And in particular, w.l.o.g. we suppose
$\vec{u}$ points in the direction of $x$-axis:
$$
  \vec{u}=\langle 1, 0, 0 \rangle, \quad \vec{v}=\langle a, \sqrt{1-a^2}, 0 \rangle, \quad \vec{w}=\langle c_1, c_2, c_3 \rangle \in\mathbb{S}^2, \quad
  -1\leq a\leq 1.
$$
Correspondingly, we have
\begin{align*}
\cos\angle(\vec{u}, \vec{v})&=a, \quad \cos\angle(\vec{u}, \vec{w})=c_1,\quad \cos\angle(\vec{v}, \vec{w})=ac_1+\sqrt{1-a^2}c_2, \\
\sin\angle(\vec{u}, \vec{w})&=\sqrt{1-c_1^2},\quad \sin\angle(\vec{v}, \vec{w})=\sqrt{1-2a\sqrt{1-a^2}c_1c_2-a^2c_1^2-(1-a^2)c_2^2}.
\end{align*}
To prove \eqref{triangle-inequality-2}, it is equivalent to show
$$
  \cos\angle(\vec{u}, \vec{v})\geq\cos\big(\angle(\vec{u}, \vec{w})+\angle(\vec{v}, \vec{w})\big).
$$
That is,
$$
  a\geq c_1(ac_1+\sqrt{1-a^2}c_2)-\sqrt{1-c_1^2}\sqrt{1-2a\sqrt{1-a^2}c_1c_2-a^2c_1^2-(1-a^2)c_2^2},
$$
which is equivalent to
\begin{equation}\label{triangle-inequality-3}
  \sqrt{1-c_1^2}\sqrt{1-2a\sqrt{1-a^2}c_1c_2-a^2c_1^2-(1-a^2)c_2^2}\geq a(c_1^2-1)+\sqrt{1-a^2}c_1c_2.
\end{equation}
If the R.H.S. of \eqref{triangle-inequality-3} is non-positive, the proof is complete. Otherwise, it is equivalent to prove
\begin{align*}
&\qquad(1-c_1^2)\big[1-2a\sqrt{1-a^2}c_1c_2-a^2c_1^2-(1-a^2)c_2^2\big]\\
&\qquad\qquad\geq a^2(1-c_1^2)^2+(1-a^2)c_1^2c_2^2-2a\sqrt{1-a^2}(1-c_1^2)c_1c_2\\
&\Leftrightarrow \quad (1-c_1^2)(1-a^2c_1^2)-(1-a^2)c_2^2\geq a^2(1-c_1^2)^2 \\
&\Leftrightarrow \quad (1-c_1^2-c_2^2)(1-a^2)\geq 0,
\end{align*}
which is obviously true.
\end{proof}

Based on Lemma \ref{lemma-grad} and Lemma \ref{lemma-H2}, we show
that the $H^2$-norm of $Q_n$ can be estimated in terms of
$\|\partial\E_n(Q_n(t))\|$ uniformly in $n\in\mathbb{N}$.

\begin{proposition}\label{lemma-H2-2}
For any $n\in\mathbb{N}$, a.e. $t\in(t_0, T)$, it holds
\begin{align*}
\big\|\Delta Q_n(t,\cdot)\big\|_{L^2(\TT^3)}\leq
C_L\big(\big\|\partial\E_n(Q_n(t,\cdot))\big\|_{L^2(\TT^3)}+2\alpha\big\|Q_n(t,\cdot)\big\|_{L^2(\TT^3)}\big),
\end{align*}
where $C_L$ is defined in \eqref{C-L}.
\end{proposition}
\begin{proof}
By Lemma \ref{lemma-grad}, we have that (up to a subsequence) $\|\partial{\E}_n(Q_n(t,\cdot))\|$
is convergent to $\|\partial\E(Q(t,\cdot))\|$ for almost every fixed
$t\in [t_0, T]$.
Hence for any fixed $t\in [t_0, T]$ (after
removing a set of measure zero), for any $n\in\mathbb{N}$ it holds
Recall \eqref{energy-sequence}, we have for every  $n\geq 1$ and almost every   $t\in (t_0,T)$ that
\begin{align}\label{bound-a.e.}
&\Big\|\partial\G(Q_n(t,\cdot))+\psi_n'(Q_n(t,\cdot))-\frac13\tr(\psi_n')(Q_n(t,\cdot))\mathbb{I}_3\Big\|_{L^2(\TT^3)}\non\\
&\leq\big\|\partial\E_n(Q_n(t,\cdot))\big\|_{L^2(\TT^3)}
+2\alpha\big\|Q_n(t,\cdot)\big\|_{L^2(\TT^3)}.\end{align}
 On the other hand,  by \noindent$(\textbf{M6})$, and $\{Q_n\}\in L^\infty(0, T; H^1(\TT^3))$, we know that
$$
 \Big\|\psi_n'(Q_n(t,\cdot))-\frac13\tr(\psi_n')(Q_n(t,\cdot))\mathbb{I}_3\Big\|_{L^2(\TT^3)}\leq \lambda_n\|Q_n(t,\cdot)\|_{L^2(\TT^3)}+\Lambda_n|\TT^3|<+\infty
$$
which together with \eqref{bound-a.e.} implies that
$$
\big\|\partial\G(Q_n(t,\cdot))\big\|_{L^2(\TT^3)}<\infty.
$$
By lemma \ref{lemma-H2}, we conclude that
$Q_n(t,\cdot)\in H^2(\TT^3)$, a.e. $t\geq 0$ for any $n\in\mathbb{N}$.

\smallskip

In the rest of this lemma, for simplicity $Q_n(\cdot, t)$ is abbreviated by $Q_n$, and w.l.o.g. we assume $\|\Delta{Q}_n\|_{L^2(\TT^3)}>0$. By \eqref{coefficient-assumption} and Lemma \ref{lemma-H2},
we have
\begin{align*}
\frac{\int_{\TT^3}\(\partial\G(Q_n): -\Delta{Q}_n\)\,dx}{\|\partial\G(Q_n)\|_{L^2}\|\Delta{Q}_n\|_{L^2}}
\geq\frac{2(L_1-|L_2+L_3|)\|\Delta{Q}_n\|_{L^2}^2}{2(L_1+|L_2+L_3|)\|\Delta{Q}_n\|_{L^2}^2}
\geq \frac{L_1-|L_2+L_3|}{L_1+|L_2+L_3|}>0.
\end{align*}
Hence the angle between $\partial\G(Q_n)$ and $-\Delta{Q}_n$ is bounded by
$$
 \angle\big(\partial\G(Q_n),-\Delta{Q}_n\big)\leq\arccos\frac{L_1-|L_2+L_3|}{L_1+|L_2+L_3|}.
$$
Meanwhile, by \noindent$(\textbf{M1})$ we know that
$$
  \int_{\TT^3}\(\partial\psi_n(Q_n): -\Delta{Q}_n\)\,dx\geq 0,
$$
which together with Lemma \ref{lemma-triangle-inequality} gives
\begin{align*}
\angle\big(\partial\G(Q_n),\partial\psi_n(Q_n)\big)&\leq\angle\big(\partial\G(Q_n),-\Delta{Q}_n\big)+\angle\big(\partial\psi_n(Q_n),-\Delta{Q}_n\big)\\
&\leq\frac{\pi}{2}+\arccos\frac{L_1-|L_2+L_3|}{L_1+|L_2+L_3|}.
\end{align*}
As a consequence, we obtain that
\begin{align}\label{trigonometry-inequlity}
&\int_{\TT^3}\(\partial\G(Q_n):\partial\psi_n(Q_n)\)\,dx\\
&\geq\cos\Big(\frac{\pi}{2}+\arccos\frac{L_1-|L_2+L_3|}{L_1+|L_2+L_3|}\Big)\|\partial\G(Q_n)\|_{L^2(\TT^3)}\|\partial\psi_n(Q_n)\|_{L^2(\TT^3)}\nonumber\\
&=-\frac{2\sqrt{L_1|L_2+L_3|}}{L_1+|L_2+L_3|}\|\partial\G(Q_n)\|_{L^2(\TT^3)}\|\partial\psi_n(Q_n)\|_{L^2(\TT^3)}.
\end{align}
Note that $2\sqrt{L_1|L_2+L_3|}/(L_1+|L_2+L_3|)<1$ by \eqref{coefficient-assumption}.

In all, after combining Lemma \ref{lemma-H2}, \eqref{bound-a.e.}, \eqref{trigonometry-inequlity} and the Cauchy-Schwarz inequality we conclude that
\begin{align}
&\big(\|\partial\E_n(Q_n)\|_{L^2(\TT^3)}+2\alpha\|Q_n\|_{L^2(\TT^3)}\big)^2\non\\
&\geq \|\partial\G(Q_n)\|_{L^2(\TT^3)}^2+\|\partial\psi_n(Q_n)\|_{L^2(\TT^3)}^2-\frac{4\sqrt{L_1|L_2+L_3|}}{L_1+|L_2+L_3|}\|\partial\G(Q_n)\|_{L^2(\TT^3)}\|\partial\psi_n(Q_n)\|_{L^2(\TT^3)}\non\\
&\geq\frac{L_1+|L_2+L_3|-2\sqrt{L_1|L_2+L_3|}}{L_1+|L_2+L_3|}\|\partial\G(Q_n)\|_{L^2(\TT^3)}^2\non\\
&\geq\frac{L_1+|L_2+L_3|-2\sqrt{L_1|L_2+L_3|}}{L_1+|L_2+L_3|}4(L_1-|L_2+L_3|)^2\|\Delta{Q}_n\|_{L^2(\TT^3)}^2
\end{align}
The proof is complete.
\end{proof}

Summing up the results established in this subsection, we are
ready to establish the improved regularity properties of the unique
solution $Q$ obtained in Theorem \ref{main-theorem-1}.

\begin{proof}[Proof of Theorem \ref{main-theorem-2}]

$\newline$

\noindent\textbf{Part 1: uniform-in-time bound \eqref{uniform-H2-bound}}
It is proved in \eqref{convergence-cite} that (up to a
subsequence) $Q_n(t)\rightarrow Q(t,\cdot)$ in $C\big([t_0, T];
L^2(\TT^3)\big)$. Moreover, it follows from Proposition \ref{lemma-H2-2}
that for a.e. $t\in (t_0,T)$, $\|\Delta Q_n(t)\|_{L^2(\TT^3)}$ is (up to a subsequence)
uniformly bounded in $n\in\mathbb{N}$. Hence $\Delta Q(t,\cdot)\in
L^2(\TT^3)$ a.e., and from Lemma \ref{lemma-grad} we can further make
the following estimates
\begin{align}
\|\Delta Q(t,\cdot)\|_{L^2(\TT^3)}&\leq \liminf_{n\rightarrow\infty}\|\Delta Q_n(t,\cdot)\|_{L^2(\TT^3)}\nonumber\\
&\leq\liminf_{n\rightarrow\infty}C_L\big(\|\partial\E_n(Q(t,\cdot))\|_{L^2(\TT^3)}+2\alpha\|Q_n(t,\cdot)\|_{L^2(\TT^3)}\big)\nonumber\\
&=C_L\big(\|\partial\E(Q(t,\cdot))\|_{L^2(\TT^3)}+2\alpha\|Q(t,\cdot)\|_{L^2(\TT^3)}\big), \label{H2-t}
\end{align}
where $C_L$ is defined in \eqref{C-L}.

W.l.o.g. we assume $T=t_0+100$. In view of Lemma \eqref{H2-t}, it suffices to provide
the  $L^\infty$-bound for $\|\partial\E(Q(t,\cdot))\|_{L^2(\TT^3)}^2$.
By equation \eqref{energy identity} we have
\begin{equation}
\int_{t_0}^{+\infty}\|\partial\E(Q(t,\cdot))\|_{L^2(\TT^3)}^2\,\ud{t}
\leq \E(Q(t_0))-\inf\E.
\end{equation}
 Thus for any $n\in\mathbb{N}$ it holds
\begin{equation}
\int_{t_0+n}^{t_0+n+1}\|\partial\E(Q(t,\cdot))\|_{L^2(\TT^3)}^2\,\ud{t}
\leq\E(Q(t_0))-\inf\E,
\end{equation}
 hence there exists a set of positive measure $A_n \subset [t_0+n,t_0+n+1]$ such that
\begin{equation}\label{energy-estimate}
\|\partial\E(Q(t,\cdot))\|_{L^2(\TT^3)}^2 \leq
\E(Q(t_0))-\inf\E+1,\quad\forall\,t\in A_n.
\end{equation}
For any $n\in\mathbb{N}$, $n\leq 99$, let us choose a time $s_n\in A_n$, and for
the sake of convenience we set $s_0=t_0$. Obviously $\{s_n\}_{n\leq 99}$ is monotone increasing, and $s_{n+1}-s_n\leq 2$.

\smallskip

Let us then consider the gradient flow sequence
\begin{equation}
\begin{cases}
\partial_tP_n\in -\partial\E(P_n),\\
P_n(0,\cdot)=Q(s_n,\cdot),
\end{cases}
\end{equation}
subject to periodic boundary condition, whose solution $P_n$ is given by the time shift $P_n(t,\cdot):=Q(s_n+t,\cdot)$,
$\forall\,t\geq 0$. By Proposition \ref{prop-energy-decay}, the
function $t\mapsto e^{-2\alpha{t}}\|\partial\E(P_n(t))\|_{L^2(\TT^3)}$ is
nonincreasing, thus together with \eqref{energy-estimate} we have
\begin{align*}
\|\partial\E(Q(t+s_n,\cdot))\|_{L^2(\TT^3)}^2
&=\|\partial\E(P_n(t,\cdot))\|_{L^2(\TT^3)}^2\leq e^{4\alpha(s_{n+1}-s_n)}\|\partial\E(P_n(0,\cdot))\|_{L^2(\TT^3)}^2\\
&\leq e^{8\alpha}\|\partial\E(Q(s_n,\cdot))\|_{L^2(\TT^3)}^2\leq e^{8\alpha}(\E(Q(t_0))-\inf\E+1),
\end{align*}
for a.e. $t\in
(0,s_{n+1}-s_n)$. Repeating this argument for all $n$ finally gives
$$
\|\partial\E(Q(t,\cdot))\|_{L^2(\TT^3)}\leq
e^{4\alpha}\sqrt{\E(Q(t_0))-\inf\E+1}, \qquad\mbox{for a.e. } t\in (t_0, t_0+99).
$$
In addition, we recall that all eigenvalues of $Q(t,\cdot)$ are in
$(-1/3,2/3)$ for a.e. $t\in (t_0, t_0+99)$, hence $\|Q(t,\cdot)\|_{H^2(\TT^3)}$ is
uniform-in-time bounded in $(t_0, t_0+99)$. Thus \eqref{uniform-H2-bound} could be proved by iteration.

\bigskip

\noindent\textbf{Part 2: strict physicality \eqref{uniform-physicality}}.
Now that $Q$ is a solution of
$$
 \partial_tQ\in -\partial\G(Q(t,\cdot))-\psi '(Q(t,\cdot))+\frac{\mathrm{tr}(\psi '(Q(t,\cdot)))}{3}\mathbb{I}_3
+2\alpha{Q}(t,\cdot) \qquad\forall\, t >t_0,
$$
with $Q\in L^\infty(t_0, +\infty;
H^2)$, let us take the inner
product with $-\Delta Q(t,\cdot)$. Then it gives
\begin{align}\label{ineq}
\dfrac{1}{2}\dfrac{d}{d{t}}\|\nabla{Q}(t,\cdot)\|^2_{L^2(\TT^3)}
&=\langle\partial\G(Q(t,\cdot)),\Delta{Q}(t,\cdot)\rangle_{L^2(\TT^3)}+\langle\psi '(Q(t,\cdot)),\Delta{Q}(t,\cdot)\rangle_{L^2(\TT^3)}\non\\
&\qquad+2\alpha\|\nabla{Q}(t,\cdot)\|^2_{L^2(\TT^3)}.
\end{align}

Note that by \eqref{coefficient-assumption-3}, Lemma \ref{lemma-H2},
and the Poincar\'e inequality, we have
\begin{align}\label{ineq-1}
\langle\partial\G(Q(t,\cdot)), \Delta{Q}(t)\rangle_{L^2(\TT^3)}
&\leq-2(L_1-|L_2+L_3|)\|\Delta{Q}(t,\cdot)\|_{L^2(\TT^3)}^2 \non\\
&\leq-\dfrac{2(L_1-|L_2+L_3|)}{C_{\TT^3}^2}\|\nabla{Q}(t,\cdot)\|_{L^2(\TT^3)}^2
\end{align}
On the other hand,
\begin{align}\label{ineq-2}
\langle \psi '(Q(t,\cdot)),
\Delta{Q}(t,\cdot)\rangle_{L^2(\TT^3)}&=\int_{\TT^3}\psi '(Q(t,x))\Delta{Q}(t,x)\,\ud{x}\non\\
&=-\int_{\TT^3}\psi ''(Q(t,x))|\nabla{Q}(t,x)|^2\,\ud{x} \leq 0,
\end{align}
due to the convexity of $\psi $. Inserting \eqref{ineq-1} and
\eqref{ineq-2} into \eqref{ineq}, we get
$$
\dfrac{d}{dt}\|\nabla{Q}(t,\cdot)\|^2_{L^2(\TT^3)} \leq
4\Big(-\frac{L_1-|L_2+L_3|}{C_{\TT^3}^2}+\alpha\Big)\|\nabla{Q}(t,\cdot)\|_{L^2(\TT^3)}^2,
\quad\forall\,t\geq t_0.
$$
Since $Q(t_0,\cdot)\in H^1(\TT^3)$, it follows from Gronwall's inequality
that
\begin{equation}\label{exp-decay}
\|\nabla{Q}(t,\cdot)\|^2_{L^2(\TT^3)} \leq \exp\bigg(4
\Big[-\frac{L_1-|L_2+L_3|}{C_{\TT^3}^2}+\alpha\Big] (t-t_0)\bigg)
\|\nabla{Q}(t_0,\cdot)\|^2_{L^2(\TT^3)}\qquad\forall\,t\ge t_0.
\end{equation}

Denote by $\bar{Q}(t):=|\TT^3|^{-1}\int_{\TT^3}Q(t,\cdot)\,\ud{x}$ the
mean value of $Q(t,\cdot)$ over $\TT^3$. Due to the convexity of
$\psi$ one can apply Jensen's inequality to derive
\begin{align*}
\psi(\bar{Q}(t)) &= \psi (\bar{Q}(t))|\TT^3| \leq \int_{\TT^3}\psi(Q(t,x))\,\ud{x}=
\E(Q(t))- \G(Q(t))+\alpha\|Q(t,\cdot)\|_{L^2(\TT^3)}^2 \\
&\leq\E(Q(t_0))+\alpha\sup_{s\geq t_0}\|Q(s,\cdot)\|_{L^\infty(\TT^3)}^2|\TT^3| ,
\end{align*}
since $\G\geq 0$ by Lemma \ref{lemma-lsc}. It is noted that the last
bound above is independent of $t\geq t_0$. Thus
\begin{equation}\label{uniform}
\rho_0:=\inf_{t\ge t_0} \rho(\-Q(t,\cdot))>0,
\end{equation}
where $\rho$ is define in \eqref{def-rho}.

\smallskip

Finally, by \eqref{H2-t} and the Gagliardo-Nirenberg inequality, we
obtain
\begin{align}\label{exp 2}
\|Q(t,\cdot)&-\bar{Q}(t)\|_{L^\infty(\TT^3)}\non\\
&\leq C'\|\nabla{Q}(t,\cdot)\|_{L^2(\TT^3)}^{1/2}\|\Delta{Q}(t,\cdot)\|_{L^2(\TT^3)}^{1/2}\non\\
&\leq C'\exp\bigg(\Big[-\frac{L_1-|L_2+L_3|}{C_{\TT^3}^2}+\alpha\Big](t-t_0)\bigg)\big\|\nabla{Q}(t_0,\cdot)\big\|_{L^2(\TT^3)}^{1/2}
\cdot C(t_0)^{1/2}
\end{align}
for some geometric constant $C'$, and a.e. $t\geq t_0$. This together with
\eqref{coefficient-assumption-3} implies
$\|Q(t,\cdot)-\bar{Q}(t)\|_{L^\infty(\TT^3)}$ decays uniformly to zero as
$t\rightarrow\infty$. Since the convergence in the $L^\infty$ norm
implies the uniform convergence of all eigenvalues, due to
\eqref{uniform} we conclude that there exists some $T_0>0$, such
that
$$
 \rho(Q(t,\cdot))\geq \dfrac{\rho_0}{2}, \quad\forall\,t\ge T_0.
$$
Therefore, the proof of Theorem \ref{main-theorem-2} is
finished.
\end{proof}


\section{Proof of Theorem \ref{thm-hausdorff}: Size of the contact set}
In this section we shall estimate the Hausdorff dimension of the singular set $\Sigma_t$ where the unique global solution $Q(t,x)$ to \eqref{grad flow E} touches the physical boundary
\begin{equation}
\Sigma_t:=\{x\in\TT^n\mid Q(t, x)\in \p \QQ \}.
\end{equation}
Here $\p\QQ$ is the boundary of $\QQ$ where the smallest eigenvalue of any element equals $-1/3$.

To begin with, we state \cite[Theorem 1.2]{LXZ20} which provides a lower bound of the blowup rate of $\p\psi(P)$ as $P\in\QQ$ approaches its physical boundary.

\begin{proposition}\label{prop-subdiff}
For any $P\in\QQ$, as $\lambda_1(P)\rightarrow -1/3$ it holds
\begin{equation}\label{log grow derivative}
\Big|\partial\psi(P)-\frac13\tr(\partial\psi(P))\mathbb{I}_3)\Big|\geq \frac{C_1}{\lambda_1(P)+\frac13}
\end{equation}
with the constant $C_1$ given by
\begin{align}\label{constant-C1}
C_1=\frac{\sqrt{3}}{9\sqrt{2\pi}e}\cdot\inf_{\xi\geq 0}\frac{e^{-\xi}\IO(\xi)}{e^{\frac{-\xi}{2}}\IO(\frac{\xi}{2})}>0,
\end{align}
where $\IO(\cdot)$ is the zeroth order modified Bessel function of first kind.
\end{proposition}
\begin{remark}
As was pointed out in \cite[Appendix C]{GT19} that $d(P, \partial\QQ)=\frac{\sqrt{6}}{2}\big[\lambda_1(P)+\frac13 \big]$ for any $P\in\QQ$. Together with Proposition
\ref{prop-subdiff}, it is immediate to see that there exists
a generic and suitably small constant $\delta_0>0$, such that
\begin{equation}
\Big|\partial\psi(P)-\frac13\tr(\partial\psi(P))\mathbb{I}_3)\Big|\geq\frac{\sqrt{6}C_1}{2d(P, \partial\QQ)}, \qquad\mbox{whenever }\;d(P,\partial\QQ)<\delta_0.
\end{equation}
\end{remark}

The following technical lemma is necessary.

\begin{lemma}\label{lemma-Vitali}
For any $s>0$ there exists a sequence of coverings $V_m=\{B_{i,m}\}$ of $\Sigma_t$, where $B_{i,m}=B(x_{i,m}, r_{i,m})$, such that
$$
  \displaystyle\lim_{m\rightarrow\infty}\sum_{i}r_{i,m}^s=\mathcal{H}^s(\Sigma_t)\leq\liminf_{m\rightarrow\infty}\sum_{j}\big|5r_{j,m}^\ast\big|^s
$$
Here $B_{i,m}=B(x_{i,m}, r_{i,m})$ such that $B_{i,m}\subset\{x\in\TT^n|\,d(x,\Sigma_t)\leq\delta\}$ for some $\delta>0$, and $r_{j,m}^\ast$ are the radii of the balls $B_{j,m}^\ast$, i.e., the sub-covering given by the Vitali covering lemma.
\end{lemma}
\begin{proof}
First, by recalling the definition of Hausdorff content
$$
  \mathcal{H}_{\delta}^s(\Sigma_t)\,:=\inf\Big\{\sum_{i}r_i^s:\,\Sigma_t\subset\bigcup_{i}B_i,\,\sup_{i}r_i<\delta\Big\},
$$
and the Hausdorff measure
$$
  \mathcal{H}^s(\Sigma_t)=\lim_{\delta\rightarrow 0}\mathcal{H}_\delta^s(\Sigma_t),
$$
we know that for any $\delta>0$ we have a sequence of coverings $V_m=\{B_{i,m,\delta}\}$ such that
$$
  \lim_{m\rightarrow\infty}\sum_{i}r_{i,m,\delta}^s=\mathcal{H}_{\delta}^s(\Sigma_t)
$$
Therefore, using the standard diagonal argument we can choose a sufficiently large $m=m(\delta)$ such that
$$
  0\leq\sum_{i}r_{i,m,\delta}^s-\mathcal{H}_{\delta}^s(\Sigma_t) :=\varepsilon_\delta \ll 1, \qquad\mbox{where }\;\varepsilon_\delta\rightarrow 0
  \Leftrightarrow m\rightarrow +\infty,
$$
which further gives
$$
  \lim_{m\rightarrow\infty}\sum_{i}r_{i,m}^s=\mathcal{H}^s(\Sigma_t),\qquad\mbox{where }\;r_{i,m}:=r_{i,m,\delta(m)}.
$$
W.l.o.g., we may choose $m$ to be bijective in $\delta$. Thus $V_m=\{B_{i,m}\}$ is an admissible sequence to
achieve the Hausdorff measure.

Note that $\{5B_{j,m}^\ast\}$ is another covering with balls for $\Sigma_t$, hence for each $m$, we have
$$
  \sum_{j}|5r_{j,m}^\ast|^s\geq\mathcal{H}_{\delta(m)}^s(\Sigma_t)=\sum_{i}r_{i,m,\delta(m)}^s-\varepsilon_{\delta},
$$
which implies
$$
  \liminf_{m\rightarrow\infty}\sum_{j}|5r_{j,m}^\ast|^s\geq\liminf_{m\rightarrow\infty}\mathcal{H}_{\delta(m)}^s(\Sigma_t)
  =\mathcal{H}^s(\Sigma_t).
$$
\end{proof}

Using Proposition \ref{prop-subdiff} and lemma \ref{lemma-Vitali}, now we are ready to finish the proof of Theorem \ref{thm-hausdorff}.

\bigskip

\begin{proof}[Proof of Theorem \ref{thm-hausdorff}]
First of all, for any $t_0>0$, by Theorems \ref{main-theorem-1} and \ref{main-theorem-2}, we know that the unique strong solution
$Q(t, x)$ to equation \eqref{strong solu} satisfies
\begin{equation}
Q\in L^{\infty}(t_0, +\infty; H^2(\TT^n)),\qquad\partial_tQ\in L^\infty(t_0, T_0; L^2(\TT^n)),
\end{equation}
which together with equation \eqref{strong solu} gives
$$
\partial\psi(Q)-\frac13\tr(\partial\psi(Q))\mathbb{I}_3\in L^{\infty}(t_0, T_0; L^2(\TT^n)).
$$
Here and after, we are only concerned with $t\in (0, T_0)$ that satisfies
$$
\Big\|\partial\psi(Q)-\frac13\tr(\partial\psi(Q))\mathbb{I}_3 \Big\|_{L^2(\TT^n)}<+\infty,
$$
which obviously has a full measure in $(0, T_0)$.

\medskip

\noindent\textbf{Case 1: $n=3$}

\medskip

By Sobolev embedding $H^2(\TT^3)\hookrightarrow C^{0,\frac12}(\TT^3)$ and $Q\in L^\infty(t_0, +\infty; H^2(\TT^3))$, there exists a generic constant $C_H>0$ that is independent of $t$, such that
\begin{equation}
|Q(t, x)-Q(t, y)|\leq C_H|x-y|^{\frac12},\qquad\forall x,y\in\TT^3.
\end{equation}
In particular, for any given $x$, let $x^{\perp}\in\Sigma_t$ be a projection such that
\begin{equation}
|x-x^{\perp}|=d(x, \Sigma_t):=\mbox{dist}(x,\Sigma_t),
\end{equation}
and henceforth we get
$$
  d(Q(t,x), \partial\QQ)\leq |Q(t,x)-Q(t,x^{\perp})|\leq C_H|x-x^{\perp}|^{\frac12}=C_H\sqrt{\mbox{dist}(x,\Sigma_t)}.
$$
This combined with Proposition \ref{prop-subdiff} implies for any ball $B$ and $\delta\ll 1$ that
\begin{align}\label{estimate-distance}
&\int_{\{x\in B|\,d(x,\Sigma_t)\leq\delta\}}\Big|\partial\psi(Q(t,x))-\frac13\tr(\partial\psi(Q(t,x)))\mathbb{I}_3\Big|^2\,dx\non\\
&\geq C_1^2\int_{\{x\in B|\,d(x,\Sigma_t)\leq\delta\}}d(Q(t,x), \partial\QQ)^{-2}dx \nonumber\\
&\geq \frac{C_1^2}{C_H^2}\int_{\{x\in B|\,d(x,\Sigma_t)\leq\delta\}}\frac{1}{\mbox{dist}(x,\Sigma_t)}\,dx.
\end{align}
To proceed, by applying Lemma \ref{lemma-Vitali} with $s=2$, we obtain the existence of a sequence of covering $\{B_{i,m}\}$
of $\Sigma_t$ with balls $B_{i,m}=B(x_{i,m},r_{i,m})$ such that $B_{i,m}\subset\{x\in\TT^3|\,d(x, \Sigma_T)\leq\delta\}$ and
$$
  \lim_{m\rightarrow\infty}\sum_{i}r_{i,m}^s=\mathcal{H}^s(\Sigma_t).
$$
We can assume that each such ball $B_{i,m}$ intersects $\Sigma_t$, for otherwise we can just remove it from the covering.
By Vitali's covering lemma, for each $m$ we can choose a sub-collection of mutually disjoint balls $\{B_{j,m}^\ast\}$, with
$B_{j,m}^\ast=B(x_{j,m}^\ast, r_{j,m}^\ast)$ such that
$$
  \Sigma_t\subset\bigcup_{i}B_{i,m}\subset\bigcup_{j}5B_{j,m}^\ast.
$$
On each such ball $B_{j,m}^\ast$, it follows from \eqref{estimate-distance} that
\begin{equation}\label{integral-small-ball}
\int_{B_{j,m}^\ast}\Big|\partial\psi(Q(t,x))-\frac13\tr(\partial\psi(Q(t,x)))\mathbb{I}_3\Big|^2\,dx
\geq\frac{C_1^2}{C_H^2}\int_{B_{j,m}^\ast}\frac{1}{\mbox{dist}(x,\Sigma_t)}\,dx.
\end{equation}
As each of such ball $B_{j,m}$ intersects $\Sigma_t$, we can choose an arbitrary intersection point $y_{j,m}\in B_{j,m}^\ast\cap\Sigma_t$, and
\begin{equation}\label{integral-mininum}
  \int_{B_{j,m}^\ast}\frac{1}{\mbox{dist}(x,\Sigma_t)}\,dx\geq\int_{B_{j,m}^\ast}\frac{1}{|x-y_{j,m}|}\,dx.
\end{equation}
We claim that the last integral in \eqref{integral-mininum} is minimized when $y_{j,m}\in\partial B_{j,m}^\ast$. To prove this claim, w.l.o.g. we assume $B_{j,m}^\ast$ is the unit ball $B$ centered at the origin, and we denote $y_{j,m}=y=(y_1, 0, 0)$, $0\leq y_1\leq 1$. Hence we consider
$$
  f(y):=\int_{B}\frac{1}{|x-y|}\,dx.
$$
Then
$$
  \frac{\partial{f}}{\partial{y_1}}=\int_{B}\frac{x_1-y_1}{|x-y|^3}\,dx, \qquad\mbox{where }\; x=(x_1, x_2, x_3).
$$
If $y_1=0$, by symmetry we see $\frac{\partial{f}}{\partial{y_1}}=0$. If $0<y_1\leq 1$, we denote $A=\{x\in\mathbb{R}^3|\,x_1\geq y_1\}$, and $A'$
the reflection of $A$ across $\{x\in\mathbb{R}^3|\,x_1=y_1\}$. Note that $A\cup A'$ is symmetric with respect to both $\{x\in\mathbb{R}^3|\,x_1=y_1\}$ and
the point $y=(y_1, 0, 0)$. Thus we have
$$
  \frac{\partial{f}}{\partial{y_1}}=\int_{B\setminus(A\cup A')}\frac{x_1-y_1}{|x-y|^3}\,dx+\underbrace{\int_{A\cup A'}\frac{x_1-y_1}{|x-y|^3}\,dx}_{=0}
  <0,
$$
due to the fact that $B\setminus(A\cup A')$ is entirely contained in
 $\{x\in\mathbb{R}^3|\,x_1\le y_1\}$. Hence the claim is proved.

As a consequence,
\begin{equation}
\int_{B_{j,m}^\ast}\frac{1}{|x-y_{j,m}|}\,dx\geq \frac{4\pi}{3}|r_{j,m}^\ast|^{2}.
\end{equation}
which together with \eqref{integral-small-ball} and \eqref{integral-mininum} gives
\begin{equation}
\int_{B_{j,m}^\ast}\Big|\partial\psi(Q(t,x))-\frac13\tr(\partial\psi(Q(t,x)))\mathbb{I}_3\Big|^2\,dx
\geq\tilde{C}|r_{j,m}^\ast|^{2},\qquad\mbox{where }
\,\tilde{C}:=\frac{4\pi}{3}\frac{C_1^2}{C_H^2}.
\end{equation}
Since $\{B_{j,m}^\ast\}$ are non-overlapping, after summing up the above inequality over $j$ we obtain
\begin{align}
&\Big\|\partial\psi(Q(t,x))-\frac13\tr(\partial\psi(Q(t,x)))\mathbb{I}_3\Big\|_{L^2(\TT^3)}^2\non\\
&\geq\sum_j\int_{B_{j,m}^\ast}\Big|\partial\psi(Q(t,x))-\frac13\tr(\partial\psi(Q(t,x)))\mathbb{I}_3\Big|^2\,dx\geq\tilde{C}\sum_{j}|r_{j,m}^\ast|^{2},
\end{align}
which together with Lemma \ref{lemma-Vitali} yields
\begin{equation}\label{Hausdorff-measure-estimate}
\Big\|\partial\psi(Q(t,x))-\frac13\tr(\partial\psi(Q(t,x)))\mathbb{I}_3\Big\|_{L^2(\TT^3)}^2\geq\tilde{C}\liminf_{m\rightarrow\infty}\sum_{j}|r_{j,m}^\ast|^{2}\geq\dfrac{\tilde{C}}{5^2}\mathcal{H}^{2}(\Sigma_t).
\end{equation}
Thus we conclude that $\mbox{dim}_{\mathcal{H}}(\Sigma_t)\leq 2$.

\medskip

\noindent\textbf{Case 2:} $n=2$

\medskip

In the $2D$ case the Sobolev embedding reads $H^2(\TT^2)\hookrightarrow C^{\beta}(\TT^2)$ for all $\beta\in (0, 1)$. Correspondingly we have
\begin{equation}
\|Q(t, x)-Q(t, y)\|\leq C_\beta|x-y|^{\beta},\qquad\forall x,y\in\TT^2,
\end{equation}
and \eqref{estimate-distance} is replaced by
\begin{align}
&\int_{\{x\in B|\,d(x,\Sigma_t)\leq\delta\}}\Big|\partial\psi(Q(t,x))-\frac13\tr(\partial\psi(Q(t,x)))\mathbb{I}_3\Big|^2\,dx\non\\
&\geq \frac{C_1^2}{C_\beta^2}\int_{\{x\in B|\,d(x,\Sigma_t)\leq\delta\}}\frac{1}{\mbox{dist}^{2\beta}(x,\Sigma_t)}\,dx.
\end{align}
As a consequence, applying Lemma \ref{lemma-Vitali} with $s=2-2\beta$, and using Vitali's covering argument exactly as in \textbf{Case 1}, one may replace
\eqref{Hausdorff-measure-estimate} by
\begin{align}
&\Big\|\partial\psi(Q(t,x))-\frac13\tr(\partial\psi(Q(t,x)))\mathbb{I}_3\Big\|_{L^2(\TT^2)}^2\non\\
&\geq\tilde{C}'\liminf_{m\rightarrow\infty}\sum_{j}|r_{j,m}^\ast|^{2-2\beta}
\geq\dfrac{\tilde{C}'}{5^{2-2\beta}}\mathcal{H}^{2-2\beta}(\Sigma_t),
\quad\mbox{where }\;\tilde{C}'=\frac{4\pi C_1^2}{C_\beta^2}.
\end{align}
In conclusion, we obtain $\mbox{dim}_{\mathcal{H}}(\Sigma_t)\leq 2-2\beta$ for any $\beta \in (0, 1)$. The proof is complete by the arbitrariness of $\beta \in (0, 1)$.
\end{proof}

{\it{Acknowledgements}}
Y.~N. Liu's work  is partially supported by NSF of China under Grant 11971314.
X.~Y. Lu's work is supported by his
NSERC Discovery Grant ``Regularity of minimizers and pattern formation in geometric minimization problems''. X. Xu's work is supported by the NSF grant DMS-2007157 and the Simons Foundation Grant No. 635288. We want to thank Professors Hongjie Dong, Fanghua Lin, Chun Liu, and Changyou Wang for their kind discussions.


\end{document}